\documentclass[a4paper,11pt,leqno]{article}

\usepackage[T2A]{fontenc}
\usepackage{amsfonts}
\usepackage{amsmath}
\usepackage{amssymb}
\usepackage{comment}
\usepackage[dvips]{graphicx}
\usepackage{subfigure}
\usepackage{expdlist}

\graphicspath{{pics/}}

\newtheorem{theorem}{Theorem}[section]
\newtheorem{lemma}[theorem]{Lemma}

\newtheorem{prop}[theorem]{Proposition}
\newtheorem{remark}[theorem]{Remark}

\numberwithin{equation}{section}

\newcommand{\set}[1]{\{ #1 \}} 
\newcommand{\setst}[2]{\{ #1 \mid #2 \}} 

\newcommand{\abs}[1]{| #1 |}
\newcommand{\val}[1]{\mathop{\rm val}\nolimits ( #1 )}
\newcommand{\supp}[1]{\mathop{\textrm{supp}}\nolimits (#1)}

\providecommand{\R}{\mathbb{R}}
\providecommand{\Z}{\mathbb{Z}}
\providecommand{\Q}{\mathbb{Q}}

\renewcommand{\vec}[1]{\overrightarrow{#1}}
\renewcommand{\phi}{\varphi}
\renewcommand{\tilde}{\widetilde}
\renewcommand{\hat}{\widehat}
\renewcommand{\bar}{\overline}
\renewcommand{\div}{\mathop{\rm div}\nolimits}

\newcommand{\eps}{\varepsilon}

\newcommand{\calB}{\mathcal{B}}

\newcommand{\calP}{\mathcal{P}}

\newcommand{\dist}{{\rm dist}}

\newcommand{\deltain}{\delta^{\rm in}}
\newcommand{\deltaout}{\delta^{\rm out}}

\newcommand{\qed}{\hfill$\square$}

\newcommand{\bvec}{\overleftarrow}

\newenvironment{proof}{\par\noindent%
{\bf Proof.~\nopagebreak}}{\nobreak\enskip\qed\par\medskip}

\newenvironment{numitem}{%
   \refstepcounter{equation}%
   \begin{enumerate} \compact%
      \item[(\theequation)]
}{%
   \end{enumerate}
}

\newenvironment{nameditem}[1]{%
   \begin{itemize}%
      \rm \item[(#1)] \em%
}{%
   \end{itemize}%
}

\newcommand{\refeq}[1]{(\ref{eq:#1})} 
\newcommand{\reffig}[1]{Fig.~\ref{fig:#1}} 
\newcommand{\refprop}[1]{Proposition~\ref{prop:#1}} 
\newcommand{\refth}[1]{Theorem~\ref{th:#1}} 
\newcommand{\reflm}[1]{Lemma~\ref{lm:#1}} 
\newcommand{\refsec}[1]{Section~\ref{sec:#1}} 
\newcommand{\refssec}[1]{Subsection~\ref{ssec:#1}} 
\newcommand{\refrem}[1]{Remark~\ref{rem:#1}} 

\newcommand{\ncprob}{$N$}
\newcommand{\ncpprob}{$N_\lambda$}
\newcommand{\ecprob}{$E$}
\newcommand{\dualpprob}{$D_\lambda$}

\newcommand{\Xcomment}[1]{}

\begin{document}

\title
{
    Min-Cost Multiflows in \\
    Node-Capacitated Undirected Networks
}

\author
{
    Maxim A. Babenko
    \thanks {
        Department of Mechanics and Mathematics, Moscow State University;
        Leninskie Gory, 119991 Moscow, Russia;
        email: \texttt{max@adde.math.msu.su}.
    }
    \and
    Alexander V. Karzanov
    \thanks {
        Institute for System Analysis of the RAS;
        9, Prospect 60 Let Oktyabrya, 117312 Moscow, Russia;
        email: \texttt{sasha@cs.isa.ru}
    }
}

\date{\today}

\maketitle

 \begin{abstract}
We consider an undirected graph $G = (VG, EG)$ with a set $T \subseteq VG$ of
terminals, and with nonnegative integer capacities $c(v)$ and costs $a(v)$ of
nodes $v\in VG$. A path in $G$ is a \emph{$T$-path} if its ends are distinct
terminals. By a \emph{multiflow} we mean a function $F$ assigning to each
$T$-path $P$ a nonnegative rational \emph{weight} $F(P)$, and a multiflow is
called \emph{feasible} if the sum of weights of $T$-paths through each node $v$
does not exceed $c(v)$. The \emph{value} of $F$ is the sum of weights $F(P)$,
and the \emph{cost} of $F$ is the sum of $F(P)$ times the cost of $P$ w.r.t.
$a$, over all $T$-paths $P$.

Generalizing known results on edge-capacitated multiflows, we show that the
problem of finding a minimum cost multiflow among the feasible multiflows
of maximum possible value admits \emph{half-integer} optimal primal and dual
solutions. Moreover, we devise a strongly polynomial algorithm for finding such
optimal solutions.
 \end{abstract}

\bigskip
\noindent \emph{Keywords}: minimum cost multiflow, bidirected graph,
skew-symmetric graph

\newpage

\section{Introduction}

\subsection{Multiflows}

For a function $\phi \colon X \to \R_+$ and a subset $A \subseteq X$, we write
$\phi(A)$ to denote $\sum_{x \in A} \phi(x)$. The \emph{incidence vector} of
$A$ in $\R^X$ is denoted by~$\chi^A$, i.e. $\chi^A(e)$ is 1 for $e \in A$ and 0
for $e\in X - A$ (usually $X$ will be clear from the context). When $A$ is a
multiset, $\chi^A(e)$ denotes the number of occurrences of $e$ in $A$.

In an undirected graph $G$, the sets of nodes and edges are denoted by $VG$ and
$EG$, respectively. When $G$ is a directed graph, we speak of arcs rather than
edges and write $AG$ instead of $EG$. A similar notation is used for paths,
cycles, and etc.

A \emph{walk} in $G$ is meant to be a sequence $(v_0,e_1,v_1, \ldots,e_k,v_k)$,
where each $e_i$ is an edge (or arc) and $v_{i-1},v_i$ are its endnodes; when
$G$ is a digraph, $e_i$ is directed from $v_{i-1}$ to $v_i$. Edge-simple (or
arc-simple) walks are called \emph{paths}.

We consider an undirected graph~$G$ and a distinguished subset $T \subseteq VG$
of nodes, called \emph{terminals}. Nodes in $VG - T$ are called \emph{inner}. A
\emph{$T$-path} is a path $P$ in $G$ whose endnodes are distinct terminals; we
usually assume that all the other nodes of $P$ are inner. The set of $T$-paths
is denoted by $\calP$. A \emph{multiflow} is a function $F \colon \calP \to
\Q_+$. Equivalently, one may think of $F$ as a collection
 \begin{equation}
\label{eq:mf_coll}
    \set{(\alpha_1, P_1), \ldots, (\alpha_n, P_m)}
\end{equation}
(for some $m$), where the $P_i$ are $T$-paths and the $\alpha_i$ are
nonnegative rationals, called \emph{weights} of paths. Sometimes (e.g.,
in~\cite{IKN-98}) such a multiflow $F$ is called \emph{free} to emphasize that
\emph{all} pairs of distinct terminals are allowed to be connected by flows.
The \emph{value} $\val{F}$ of $F$ is $\sum_P F(P)$. For a node $v$, define
   \begin{equation}
\label{eq:node_load}
    \hat F(v) := \sum \left( F(P) \colon P \in \calP, v \in VP \right);
 \end{equation}
the function $\hat F$ on $VG$ is called the \emph{(node) load function}. Let $c
\colon VG \to \Z_+$ be a nonnegative integer function of \emph{node
capacities}. We say that $F$ is \emph{feasible} if $\hat F(v) \le c(v)$ for all
$v \in VG$.

Suppose we are given, in addition, a function $a \colon VG \to \Z_+$ of
\emph{node costs}. Then the \emph{cost} $a(F)$ of a multiflow $F$ is the sum
$\sum_P a(P)F(P)$, where $a(P)$ stands for the cost $a(VP)$ of a path $P$.

In this paper we consider the following problem:
 \begin{nameditem}{\ncprob}
Given $G, T, c, a$ as above, find a multiflow~$F$ of minimum possible cost
$a(F)$ among all feasible multiflows of maximum possible value.
 \end{nameditem}

\subsection{Previous results}

When $\abs{T} = 2$, ~(\ncprob) turns into the undirected min-cost max-flow
problem under node capacities and costs, having a variety of applications; see,
e.g.,~\cite{FF-62,law-76}. It admits integer optimal primal and dual
solutions~\cite{FF-62}.

In the special case $a\equiv 0$, we are looking simply for a feasible
multiflow of maximum value. Such a problem has half-integer optimal primal and
dual solutions, due to results of Pap~\cite{pap-07} and Vazirani~\cite{vaz-01},
respectively. Also it is shown in~\cite{pap-07} that the problem is solvable in
strongly polynomial time by using the ellipsoid method.

An edge-capacitated version of (\ncprob) has been well studied. In this
version, denoted by~(\ecprob), $c$ and $a$ are functions on $EG$ rather than $VG$.
For a multiflow~$F$, its \emph{edge load function} is defined similarly to \refeq{node_load}:
 \begin{equation}
\label{eq:edge_load}
    \hat F(e) := \sum \left( F(P) \colon P \in \calP, e \in EP \right)
    \quad \mbox{for all $e \in EG$},
 \end{equation}
and its cost is defined to be $\sum_P a(EP)F(P)$. Problem~(\ecprob) is reduced
to (\ncprob) by adding an auxiliary node on each edge, but no converse
reduction is known.

An old result is that (\ecprob) has a half-integer optimal
solution~\cite{kar-79}. Also it is shown in~\cite{kar-94} that~(\ecprob) has a
half-integer optimal dual solution and that half-integer primal and dual
optimal solutions can be found in strongly polynomial time by using the
ellipsoid method. A ``purely combinatorial'' weakly polynomial algorithm, based
on cost and capacity scaling, is devised in~\cite{GK-97}.

In the special case of (\ecprob) with $a\equiv 0$, the half-integrality results
are due to Lov\'asz~\cite{lov-76} and Cherkassky~\cite{cher-77}, and a strongly
polynomial combinatorial algorithm is given in~\cite{cher-77} (see
also~\cite{IKN-98} for faster algorithms).

\subsection{New results}

In this paper we prove that (\ncprob) always admits a half-integer optimal
primal and dual solutions. In particular, this implies all half-integrality
results mentioned in the previous subsection.

Similar to \cite{kar-94}, we introduce a \emph{parametric} generalization of
(\ncprob), study properties of geodesics (shortest $T$-paths with respect to some
length function), and reduce the parametric problem to a certain
single-commodity flow problem. However, the details of this construction are
more involved. In particular, the reduced problem concerns integer flows in a
\emph{bidirected} graph.

The second goal is to explore the complexity of (\ncprob). We show that
half-integer optimal primal and dual solutions to the parametric problem (and
therefore to~(\ncprob)) can be found in strongly polynomial time by using the
ellipsoid method.

\section{Preliminaries}
\label{sec:prelim}

\subsection{Parametric problem and its dual}

Instead of (\ncprob), it is convenient to consider a more general problem,
namely:
\begin{nameditem}{\ncpprob}
    Given $G, T, c, a$ as in (\ncprob) and, in addition, $\lambda \in \Z_+$,
    find a feasible multiflow~$F$ maximizing the
    objective function $\Phi(F,a,\lambda) := \lambda \cdot \val{F} - a(F)$.
 \end{nameditem}
We will prove the following
 \begin{theorem}
\label{th:primal_halfint}
    For any $\lambda \in \Z_+$, problem (\ncpprob) has a half-integer optimal solution.
 \end{theorem}
(Note that $\Phi(F, qa, q\lambda)= q \cdot \Phi(F,a,\lambda)$ for any multiflow
$F$ and $q\in\Q_+$. Therefore, the optimality of a multiflow in the parametric
problem preserves when both $a$ and $\lambda$ are multiplied by the same
positive factor $q$. This implies that the theorem is generalized to arbitrary
$a:VG\to\Q_+$ and $\lambda\in\Q_+$ (but keeping the integrality of $c$).
However, we prefer to deal with integer-valued $a$ and $\lambda$ in what
follows.)

By standard linear programming arguments, (\ncprob) and (\ncpprob) become
equivalent when $\lambda$ is large enough (moreover, the existence of a
half-integer optimal solution for (\ncpprob) easily implies that taking
$\lambda := 2c(VG)a(VG) + 1$ is sufficient).

Problem (\ncpprob) can be viewed as a linear program with variables $F(P) \in
\Q_+$ assigned to $T$-paths $P$. Assign to a node $v \in VG$ a variable $l(v)
\in \Q_+$. Then the linear program dual to (\ncpprob) is:
\begin{nameditem}{\dualpprob}
    Minimize $c \cdot l$ provided that the following
    holds for every $T$-path~$P$:
    \begin{equation}
    \label{eq:dual_ineq}
        l(P) \ge \lambda - a(P).
    \end{equation}
\end{nameditem}

\subsection{Translating to edge lengths}

The above dual problem (\dualpprob) involves lengths of paths (namely, $l(P)$
and $a(P)$) determined by ``lengths'' of nodes ($l$ and $a$, respectively). It
is useful to transform lengths of nodes into lengths of edges. To do so, for $w
\colon VG \to \Q_+$, we define the function $\bar w$ on $EG$ by
  \begin{equation}  \label{eq:wbar}
  \bar w(e):=\alpha_u w(u)+\alpha_v w(v)\qquad \mbox{for $e=uv\in EG$},
   \end{equation}
where $\alpha_x:=\frac12$ if $x \in VG-T$, and $\alpha_x := 1$ if $x \in T$. This provides
the correspondence
\begin{equation}
    \label{eq:node_to_edge}
    w(P) = \bar w(P) \qquad \mbox{for each $T$-path $P$}
\end{equation}
(where $w(P)$ stands for $w(VP)$, and $\bar w(P)$ for $\bar w(EP)$). For $a,l$
as above, define
  \begin{equation} \label{eq:ell}
  \ell :=\bar l + \bar a.
  \end{equation}
Let $\dist_\ell(u,v)$ denote the $\ell$-distance between vertices $u$ and $v$,
i.e. the minimum $\ell$-length $\ell(P)$ of a $u$--$v$ path $P$ in $G$. Then,
in view of~\refeq{node_to_edge} and~\refeq{ell}, the constraints in
(\dualpprob) can be rewritten as
\begin{equation}
    \label{eq:dual_dist}
    \dist_\ell(s,t) \ge \lambda \qquad \mbox{for all $s, t \in T$, $s \ne t$.}
\end{equation}

By the linear programming duality theorem applied to (\ncpprob) and (\dualpprob), a
feasible multiflow $F$ and a function $l \colon VG \to \Q_+$ satisfying
\refeq{dual_dist} are optimal solutions to (\ncpprob) and (\dualpprob),
respectively, if and only if the following (complementary slackness conditions)
hold:
\begin{numitem}
    \label{eq:cs_geodesic}
    if $P$ is a $T$-path and $F(P) > 0$, then $\ell(P) = \lambda$; in particular,
    $P$ is $\ell$-shortest;
\end{numitem}
\begin{numitem}
    \label{eq:cs_saturated}
    if $v \in VG$ and $l(v) > 0$, then $v$ is \emph{saturated} by $F$, i.e. $\hat F(v) = c(v)$.
\end{numitem}

In the rest of the paper, to simplify technical details, we will always assume
that the input costs~$a$ of all nodes are \emph{strictly positive}. Then the
edge lengths $\ell$ defined by~\refeq{ell} are strictly positive as well. This
assumption will not lead to loss of generality in essence, since the desired
results for a nonnegative input cost function~$a$ can be obtained by applying a
perturbation technique in spirit of~\cite[pp.~320--321]{kar-94} (by replacing
$a$ by an appropriate strictly positive cost function).

\subsection{Geodesics}
\label{ssec:geodesics}

Condition \refeq{cs_geodesic} motivates the study of the structure of
$\ell$-shortest $T$-paths in~$G$. To this aim,
set $p := \min \setst{\dist_\ell(s,t)}{s, t \in T, s \ne t}$.
A $T$-path $P$ such that $\ell(P) = p$ is called an \emph{$\ell$-geodesic} (or just \emph{geodesic} if $\ell$ is clear form the context).
When a multiflow $F$ in $G$ is given
as a collection~\refeq{mf_coll} in which all paths $P_i$ are $\ell$-geodesics,
we say that $F$ is an \emph{$\ell$-geodesic} multiflow.

Next we utilize one construction from \cite{kar-79,kar-94}, with minor changes.
Consider a node $v \in VG$. Define the \emph{potential} $\pi(v)$ to be the
$\ell$-distance from $v$ to the nearest terminal, i.e. $\pi(v) := \min
\setst{\dist_\ell(v,t)}{t \in T}$. Set $VG_\ell := \setst{v \in VG}{\pi(v) \le
\frac12 p}$ (in particular, $T \subseteq VG_\ell$). For $s \in T$, define $V^s
:= \setst{v \in VG}{\dist_\ell(s,v) < \frac12 p}$. Also define $V^\natural :=
\setst{v \in VG}{\pi(v) = \frac12 p}$. We refer to $V^s$ as the \emph{zone} of
a terminal $s\in T$, and to $V^\natural$ as the set of \emph{central} nodes
(w.r.t. $\ell$). The sets $V^s$ ($s \in T$) and $V^\natural$ are pairwise
disjoint and give a partition of~$VG_\ell$.

The following subset of edges is of importance:
\begin{align*}
    EG_\ell :=
    \setst{uv \in EG}{ \exists\, s \in T\,\colon\, u \in V^s, v \in V^s \cup V^\natural,
    \abs{\pi(u) - \pi(v)} = \ell(uv)}  \\
    \cup \setst{uv \in EG}{ \exists\, s, t \in T, s \ne t\,\colon\,u \in V^s, v \in V^t,
    \pi(u) + \pi(v) + \ell(uv) = p}.
\end{align*}

One can see that the subgraph $G_\ell := (VG_\ell, EG_\ell)$ of $G$ contains
all $\ell$-geodesics. Moreover, a straightforward examination shows that the
structure of $\ell$-geodesics possesses the properties as expressed in the
following lemma (which is, in fact, a summary of Claims~1--3 from \cite{kar-94}
and uses the strict positivity of $\ell$).
\begin{lemma}
\label{lm:geodesics}
    Let $P$ be an $\ell$-geodesic running from $s\in T$ to $t\in T$.
    Then $P$ is contained in $G_\ell$ and exactly one of the following takes place:
    \begin{enumerate} \compact
        \item
        $P$ contains no central nodes and can be represented as the concatenation
        $P_1 \circ (u, e, v) \circ P_2$, where $u \in V^s$, $v \in V^t$, $s \ne t$, and
        $e \in EG_\ell$.

        \item
        $P$ contains exactly one central node $w \in V^\natural$ and can be
        represented as the concatenation
        $P = P_1 \circ (u, e_1, w, e_2, v) \circ P_2$, where $u\in V^s$, $v\in V^t$, $s \ne t$,
        and $e_1, e_2 \in EG_\ell$.
    \end{enumerate}
    In both cases, parts $P_1$ and $P_2$ are contained in the induced subgraphs $G_\ell[V^s]$
    and $G_\ell[V^t]$, respectively. The potentials $\pi$ are strictly increasing
    as we traverse $P_1$ from $s$ to $u$, and strictly decreasing as we traverse
    $P_2$ from $v$ to $t$.

    Conversely, any $T$-path in $G_\ell$ obeying the above properties is an
    $\ell$-geodesic.
\end{lemma}

\section{Primal half-integrality}

\subsection{Auxiliary bidirected graph}
\label{ssec:aux_graph}

In this subsection we introduce an auxiliary bidirected graph, which will be
the cornerstone of our approach both for proving half-integrality results and
for providing a polynomial-time algorithm.

Given $G$, $T$, $c$, $a$ and $\lambda$ as above, let $l$ be an optimal solution
to (\dualpprob). Form the edge lengths $\ell := \bar l + \bar a$, the potential
$\pi$, the subgraph $G_\ell$, and the sets $V^s$ ($s \in T$) and $V^\natural$,
as in \refssec{geodesics}. One may assume that $p: = \min
\setst{\dist_\ell(s,t)}{s, t \in T, s \ne t} = \lambda$ (since $p\ge \lambda$,
by~\refeq{dual_dist}, and if $p>\lambda$ then $F = 0$ is an optimal solution
to~(\ncpprob), by~\refeq{cs_geodesic}).

For further needs, we reset $c:=2c$, making all node capacities even integers.
Now our goal is to prove the existence of an \emph{integer} optimal multiflow
$F$ in problem (\ncpprob) (which is equivalent to proving the half-integrality
w.r.t. the initial $c$).
\medskip

Recall that in a \emph{bidirected} graph (or a \emph{BD-graph} for short) edges
of three types are allowed: a usual directed edge, or an \emph{arc}, that
leaves one node and enters another one; an edge directed \emph{from both} of
its ends; and an edge directed \emph{to both} of its ends
(cf.~\cite{EJ-70,sch-03}). When both ends of an edge coincide, the edge becomes
a loop. For our purposes we admit no loop entering and leaving its end node
simultaneously. Sometimes, to specify the direction of an edge $e=uv$ at one or
both of its ends, we will draw arrows above the corresponding node characters.
For example, we may write $\vec{uv}$ if $e$ is directed from $u$ to $v$ (a
usual arc), ${\vec u}\bvec{v}$ if $e$ leaves both $u,v$, $\bvec{u}\vec{v}$ if
$e$ enters both $u,v$, and $\vec{u}v$ if $e$ leaves $u$ (and either leaves or
enters $v$).

A \emph{walk} in a BD-graph is an alternating sequence
$$
    P = (s = v_0, e_1, v_1, \ldots, e_k, v_k = t)
$$
of nodes and edges such that each edge $e_i$ connects nodes $v_{i-1}$ and
$v_i$, and for $i = 1, \ldots, k-1$, the edges $e_i,e_{i+1}$ form a
\emph{transit pair} at $v_i$, which means that one of $e_i,e_{i+1}$ enters and
the other leaves~$v_i$. As before, an edge-simple walk is referred to as a
\emph{path}.

Now we associate to $G_\ell$ a BD-graph $H$ with edge capacities $c \colon EH
\to \Z_+$, as follows (see \reffig{aux_graph} for an illustration). Each
noncentral node $v \in VG_\ell - V^\natural$ generates two nodes $v^1,v^2$ in
$H$. They are connected by edge (arc) $e_v$ going from $v^1$ to $v^2$ and
having the capacity equal to $c(v)$. We say that $e_v$ \emph{inherits the
capacity} of the node~$v$. For $s \in T$, the set $\bar V^s:=\setst{v^1, v^2}{v
\in V^s}$ in $H$ is called the \emph{zone} of $s$, similar to $V^s$ in~$G$.

Consider an edge $e=uv \in EG_\ell$. Let $u, v \in V^s$ for some $s \in T$ and
assume for definiteness that $\pi(u) < \pi(v)$ (note that $\ell(uv)>0$ implies
$\pi(u)\ne\pi(v)$; this is where the strict positivity of the cost function $a$
is important). Then $e$ generates in $H$ an edge (arc) going from $u^2$ to
$v^1$, and we assign infinite capacity to it. (By ``infinite capacity'' we mean
a sufficiently large positive integer.) Now let $u \in V^s$ and $v \in V^t$ for
distinct $s, t \in T$. Then $e$ generates an infinite capacity edge
$\vec{u}^2\bvec{v}^2$ (leaving both $u^2$ and~$v^2$).

The transformation of central nodes is less straightforward. Each $w \in
V^\natural$ generates in $H$ a so-called \emph{gadget}, denoted by $\Gamma_w$.
It consists of $\abs{T} + 1$ nodes; they correspond to $w$ and the elements of
$T$ and are denoted as $\theta_w$ and $\theta_{w,s}$, $s \in T$. The edges of
$\Gamma_w$ are: a loop $e_w$ leaving $\theta_w$ (twice) and, for each $s \in
T$, an edge $e_{w,s}$ going from $\theta_{w,s}$ to $\theta_w$, called the
$s$-\emph{leg} in the gadget. Each edge in $\Gamma_w$ is endowed with the
capacity equal to $c(w)$.

Each gadget $\Gamma_w$ is connected to the remaining part of $H$ as follows.
For each edge of the form $vw$ in $G_\ell$, we know that $v \in V^s$ for some
$s \in T$ (by the construction of $G_\ell$). Then $vw$ generates an infinite
capacity edge (arc) going from $v^2$ to $\theta_{w,s}$.

Finally, we add to $H$ an extra node $q$, regarding it as the
\emph{source}, and for each $s \in T$, draw an infinite capacity edge (arc)
from $q$ to $s^1$.

\begin{figure}[tb]
    \centering
    \subfigure[Graph $G_\ell$.]{
      \includegraphics{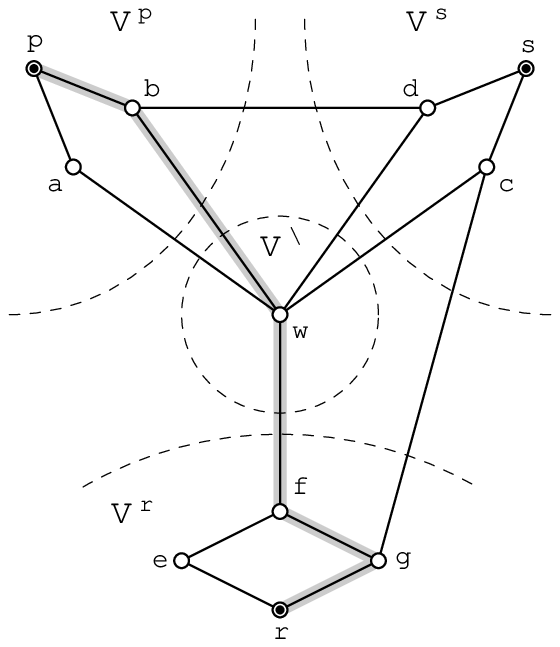}%
    }
    \hspace{1cm}%
    \subfigure[Bidirected graph $H$.]{
      \includegraphics{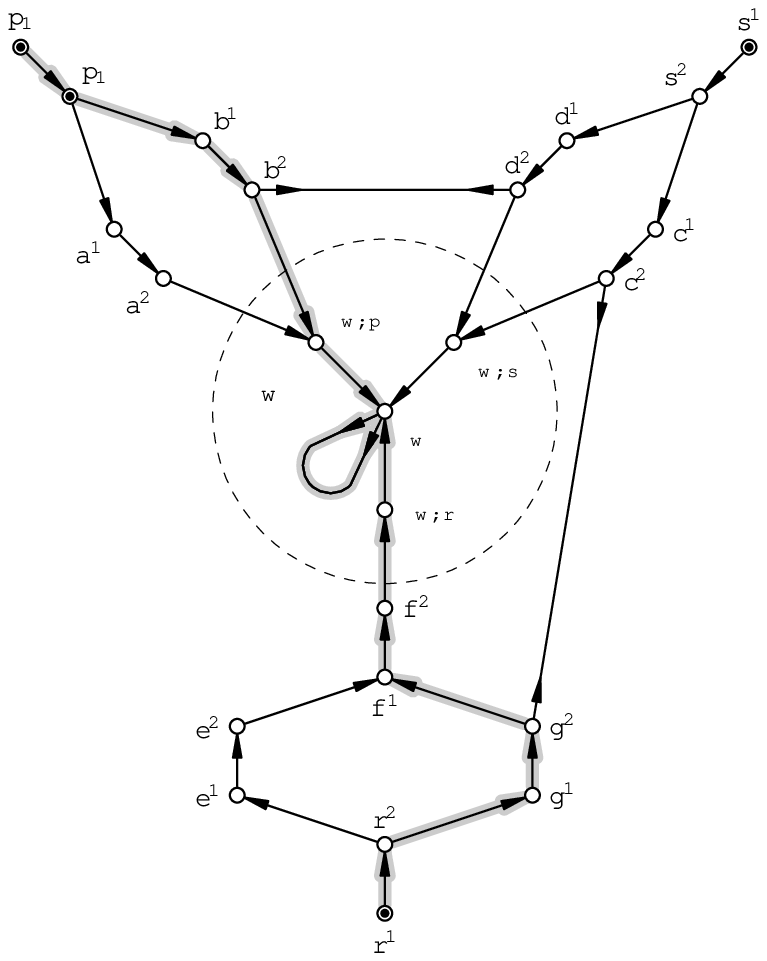}%
    }
    \caption{
        Constructing graph $H$. Here $T = \set{p,s,r}$, $V^p = \set{p,a,b}$,
        $V^s = \set{s,c,d}$, $V^r = \set{r,e,f,g}$, $V^\natural = \set{w}$.
        (The source $q$ is not shown.)
        Bidirected edges leaving one endpoint and entering the other are
        indicated by ordinary directed arcs.
        Marked are one $\ell$-geodesic $P$ and its image $\bar P$.
    }
    \label{fig:aux_graph}
\end{figure}

The obtained BD-graph $H$ captures information about the $\ell$-geodesics
in~$G$. Namely, each $\ell$-geodesic $P$ going from $s$ to $t$ induces a unique
closed $q$--$q$ walk $\bar P$ in $H$. The first and the last edges of $\bar P$
are $\vec{qs}^1$ and $\bvec{t}^1\bvec{q}$, respectively. For a noncentral node
$v$ in $P$, ~$\bar P$ traverses the edge $\vec{v}^1\vec{v}^2$. An edge $uv\in
EP$ with $\pi(u)<\pi(v)$ inside a zone induces the edge $\vec{u}^2\vec{v}^1$ in
$\bar P$. An edge $uv \in EP$ connecting different zones (if any) induces the
edge $\vec{u}^2\bvec{v}^2$ in $\bar P$. Finally, suppose $P$ traverses a
central node $w \in V^\natural$ and let $uw, wv \in EG_\ell$ be the edges of
$P$ incident to $w$. By \reflm{geodesics}, $u \in V^s$ and $v \in V^t$ for some
$s\ne t$. Then the sequence of nodes $u, w, v$ in $P$ generates the subpath in
$\bar P$ with the sequence of edges
$u^2\theta_{w,s},e_{w,s},e_w,e_{w,t},\theta_{w,t}v^2$.

The resulting walk $\bar P$ is edge-simple, so it is a closed path. Conversely,
let $Q$ be a (nontrivial) $q$--$q$ walk in $H$. One can see that $Q$ with $q$
removed is concatenated as $Q_1\circ Q'\circ Q_2$, where $Q_1$ is a directed
path within a zone $\bar V^s$, $Q_2$ is reverse to a directed path within a
zone $\bar V^t$ (with possibly $s=t$), and $Q'$ either~(i) is formed by an edge
$\vec{u}^2\bvec{v}^2$ connecting these zones (in which case $s\ne t$), or~(ii)
is the walk with the sequence of edges $e_{w,s},e_w,e_{w,t}$, for some central
node $w$ of $G_\ell$. Moreover, the image in $G$ of each of $Q_1,Q_2$ is an
$\ell$-shortest path. When $s=t$ happens in case~(ii), $Q$ traverses the edge
$e_{w,s}$ twice. In all other cases, $Q$ is edge-simple and its image in $G$ is
an $\ell$-shortest $T$-path (a $\lambda$-geodesic).

These observations show that there is a natural bijection between the
$\ell$-geodesics in $G$ and the (nontrivial) $q$--$q$ paths in $H$.
\medskip

We will refer to the BD-graph $H$ described above as the \emph{compact}
BD-graph related to $G_\ell$; it will be essentially used to devise an
efficient algorithm for solving~(\ncpprob) in \refsec{primal_alg}. Besides, in
the proof of the primal integrality (with $c$ even) in \refsec{prim_h_int}, we
will deal with a modified BD-graph. It is obtained from $H$ as above by
replicating each gadget $\Gamma_w$ into $c(w)$ copies $\Gamma_{w^i}$,
$i=1,\ldots,c(w)$, called the \emph{1-gadgets} generated by $w$. More
precisely, to construct $\Gamma_{w^i}$, we make $i$-th copy $\theta_{w^i}$ of
the node $\theta_w$, ~$i$-th copy $e_{w^i}$ of the loop $e_w$ leaving
$\theta_{w^i}$ (twice), and $i$-th copy $e_{w^i,s}$ of each leg $e_{w,s}$,
$s\in T$, where $e_{w^i,s}$ goes from $\theta_{w,s}$ to $\theta_{w^i}$ (so
$\theta_{w,s}$, $s\in T$, are the common nodes of the created 1-gadgets). All
edges in these 1-gadgets are endowed with \emph{unit} capacities.

We keep notation $H$ for the constructed graph and call it the
\emph{expensive} BD-graph related to $G_\ell$. Also we keep notation $c$ for
the edge capacities in $H$. There is a natural relationship between the
$q$--$q$ walks (paths) in both versions of $H$. The 1-gadgets created from
the same central node $w$ of $G_\ell$ are isomorphic, and for any
$i,j=1,\ldots,c(w)$, there is an automorphism of $H$ which swaps
$\theta_{w^i}$ and $\theta_{w^j}$ and is invariant on the other nodes.

\subsection{Bidirected flows}
\label{ssec:bidir_flows}

Let $\Gamma$ be a bidirected graph. Like in usual digraphs, $\deltain(v)$ and
$\deltaout(v)$ denote the sets of edges in $\Gamma$ entering and leaving $v \in
V\Gamma$, respectively. A loop $e$ at $v$, if any, is counted twice in
$\deltain(v)$ if $e$ enters $v$, and twice in $\deltaout(v)$ if $e$ leaves $v$;
hence $\deltain(v)$ and $\deltaout(v)$ are actually multisets. (Recall that we
do not allow a loop which simultaneously enters and leaves a node.)

Let $q$ be a distinguished node with $\deltain(q)=\emptyset$ in~$\Gamma$ (the
\emph{source}) and let the edges of~$\Gamma$ have integer capacities $c \colon
E\Gamma \to \Z_+$. A \emph{bidirected $q$-flow}, or a \emph{BD-flow} for short,
is a function $f \colon E\Gamma \to \Q_+$ satisfying $\div_f(v) = 0$ for all
nodes $v \in V\Gamma - \set{q}$; and the \emph{value} of $f$ is defined to be
$\div_f(q)$ (cf.~\cite{GK-04}). Here
\begin{equation}
\label{eq:div}
    \div_{f}(v) := f(\deltaout(v)) - f(\deltain(v))
\end{equation}
is the \emph{divergence} of $f$ at $v$. Note that if $e$  is a loop at $v$ then
$e$ contributes $\pm 2f(e)$ in $\div_{f}(v)$. If $f(e) \le c(e)$ for all $e \in
E\Gamma$ then $f$ is called \emph{feasible}. In addition, if $f$ is
integer-valued on all edges then we refer to $f$ as an integer bidirected
$q$-flow, or an \emph{IBD-flow}. One can see that finding a fractional (resp.
integer) BD-flow of the maximum value is equivalent to constructing a maximum
fractional (resp. integer) packing of closed $q$--$q$ walks (they leave $q$
twice).

\medskip

Return to an optimal solution $l$ to (\dualpprob), and let $\ell := \bar a +
\bar l$. Consider the (expensive or compact) BD-graph $H$ related to $G_\ell$,
and the capacity function $c$ on the edges of $H$ (constructed from the node
capacities $c$ of $G$). The above correspondence between $\ell$-geodesics in
$G$ and $q$--$q$ paths in $H$ is extended to $\ell$-geodesic multiflows in $G$
and certain $q$-flows in $H$ (where $q$ is the source in $H$ as before). More
precisely, let $F$ be a (fractional) $\ell$-geodesic multiflow in $G$
represented by a collection of $\ell$-geodesics $P_i$ and weights $\alpha_i :=
F(P_i)$, $i = 1, \ldots, m$ (cf.~\refeq{mf_coll}). Then each $P_i$ determines a
$q$--$q$ path $\bar P_i$ in $H$, and $f := \alpha_1\chi^{E\bar P_1} + \ldots +
\alpha_m\chi^{E\bar P_m}$ is a BD-flow in $H$; we say that $f$ is
\emph{generated} by $F$ (note that $\val{f}=2\val{F}$). Furthermore, $f$ is
feasible if $F$ is such, and for each central node $w\in V^\natural$, the
following relations hold:
   \begin{equation}
   \label{eq:good}
   \sum\nolimits_{s\in T} f(e_{w,s})=2f(e_w)\quad \mbox{and} \quad
   f(e_{w,s})\le f(e_w)\;\; \mbox{for each $s\in T$}.
   \end{equation}

Considering an arbitrary BD-flow $f$ in $H$, we say that $f$ is \emph{good} if
it satisfies~\refeq{good} for all $w\in V^\natural$ (here the second relation
in~\refeq{good} is important, while the first one obviously holds for any
BD-flow). The following assertion is of use.
  \begin{lemma} \label{lm:f-F}
Let $f$ be a good BD-flow in $H$. Then $f$ is generated by an
$\ell$-geodesic multiflow $F$ in $G$. Moreover, if $f$ is integral, then it is
generated by an integer $\ell$-geodesic multiflow $F$. In both cases, $F$ can
be found in $O(|EH|)$ time.
  \end{lemma}
  \begin{proof}
Suppose there is a central node $w\in V^\natural$ such that $f(e_w)>0$. Let us
say that $p\in T$ \emph{dominates} at $w$ (w.r.t. $f$) if $f(e_{w,p})=f(e_w)$.
From~\refeq{good} it follows that there exist distinct $s,t\in T$ such that
$f(e_{w,s}),f(e_{w,t})>0$ and none of $p\in T-\{s,t\}$ dominates at $w$. Choose
such $s,t$. Build in $H$ a maximal walk $Q$ starting with
$\theta_w,e_{w,s},\theta_{w,s},\ldots$ and such that $f(e)>0$ for all edges $e$
of $Q$. It is easily seen from the construction of $H$ that $Q$ is
edge-simple, terminates at $q$, and have all vertices in $\bar V^s$, except for
$\theta_w,\theta_{w,s}$. Build a similar walk (path) $Q'$ starting with
$\theta_w,e_{w,t},\theta_{w,t}$. Then the concatenation of the reverse to $Q$,
the loop $e_w$ and the path $Q'$ is a $q$--$q$ path and its image $P$ in $G$ is
an $\ell$-geodesic (from $s$ to $t$).

Assign the weight of $P$ to be the maximum number $\alpha$ subject to two
conditions: (i) $\alpha\le f(e)$ for each $e\in E\bar P$, and (ii) the flow
$f':=f-\alpha\chi^{\bar P}$ is still good. If $\alpha$ is determined by~(i), we
have $|\supp{f'}|<|\supp{f}|$ (where $\supp{\phi} := \setst{x}{\phi(x) \ne 0}$),
whereas if $\alpha$ is determined by~(ii), there appears $p\in T$ dominating at
$w$ (w.r.t. $f'$). If $f'(e_w)>0$, repeat the procedure for $f'$ and $w$,
otherwise apply the procedure to $f'$ and another $w'\in V^\natural$, and so
on. (Note that if, in the process of handling $w$, a current weight $\alpha$ is
determined by~(ii), then the weights of all subsequent paths through $e_w$ are
already determined by~(i); this will provide the desired complexity.)

Eventually, we come to a good flow $\tilde f$ with $\tilde f(e_w)=0$ for
all $w\in V^\natural$. This $\tilde f$ is decomposed into a sum of flows along
$q$--$q$ paths in a straightforward way, in $O(|EH|)$ time (like for usual
flows in digraphs). Taking together the images in $G$ of the constructed
weighted $q$--$q$ paths, we obtain a required $\ell$-geodesic multiflow $F$.
The running time of the whole process is $O(|EH|)$, and if $f$ is
integral, then the weights $\alpha$ of all paths are integral as well.
(Integrality of a current weight $\alpha$ subject to integrality of a current
flow $f$ is obvious when $\alpha$ is determined by~(i), and follows from
the fact, implied by~\refeq{good}, that for any $p\in T$, ~$\sum_{s\ne
p}f(e_{w,s})- f(e_{w,p})$ is even, when $\alpha$ is determined by~(ii).)
  \end{proof}

\begin{remark} \label{rem:compact_vs_expensive}
    In case of an expensive BD-graph $H$, any feasible IBD-flow $f$ is good.
    Indeed, for any 1-gadget $\Gamma_{w^i}$ in $H$, we have $c(e_{w^i})=1$ and, therefore, $f(e_w) \in \set{0,1}$.
    The second relation in~\refeq{good} is trivial when $f(e_w) = 0$, and follows from the
    constraints $f(e_{w,s}) \le c(e_{w,s}) = 1$ ($s\in T$) when $f(e_w) = 1$.
\end{remark}

Define the following subset of edges in $H$:
\begin{equation} \label{eq:E0}
    E_0 := \setst{e_v}{v \in VG, \;\; l(v) > 0}.
\end{equation}
For an optimal (possibly fractional) solution $F$ to~(\ncpprob) and a node
$v\in VG$ with $l(v)>0$, we have $\hat F(v) = c(v)$ (by \refeq{cs_saturated});
so the edge $e_v$ of $H$ corresponding to $v$ is
saturated by the BD-flow $f$ generated by $F$, i.e. $f(e_v) = c(e_v)$. We call the
edges in $E_0$ \emph{locked}.

Thus, the graph $H$ admits a (fractional) good feasible BD-flow
saturating the locked edges. The following strengthening is crucial.
\begin{prop}
\label{prop:aux_bidir_flow}
    There exists a good feasible IBD-flow in $H$ that saturates all locked edges.
\end{prop}

A proof of this proposition involves an additional graph-theoretic machinery
and will be given in \refsec{prim_h_int}. Assuming its validity, we immediately
obtain \refth{primal_halfint} from \reflm{f-F}.

\section{Solving (\ncpprob) in strongly polynomial time}
\label{sec:primal_alg}

In this section we devise a strongly polynomial algorithm for solving the
primal parametric problem~(\ncpprob). As before, we assume that $a$ and
$\lambda$ are integral and that the node capacities $c$ are even, so our goal
is to find an integer optimal multiflow.

The algorithm starts with computing a (fractional) optimal dual solution $l$
and constructing the BD-graph $H$ w.r.t. the length function $\ell:=\bar
a+\bar l$. Then it finds a good IBD-flow $f$ in $H$ saturating the locked
edges (assuming validity of \refprop{aux_bidir_flow}). Applying the efficient
procedure as in the proof of \reflm{f-F} to decompose $f$ into a collection of
paths with integer weights, we will obtain an integer optimal solution
to~(\ncpprob).

To provide the desired complexity, we shall work with $H$ given in the
\emph{compact} form (defined in \refssec{aux_graph}). The core of our method
consists in finding the load function of some integer optimal multiflow $F$ in
$G$ (without explicitly computing $F$ itself). This function will just generate
the desired IBD-flow in $H$. We describe the stages of the algorithm in the
subsections below.

\subsection{Constructing an optimal dual solution}
\label{ssec:dual_algo}

Problem (\dualpprob) straightforwardly reduces to a ``compact'' linear program,
as follows. Besides variables $l(v)\in\Q_+$ ($v\in VG$), assign a variable
$\phi_s(v)\in\Q$ to each terminal $s \in T$ and node $v$ (a sort of
``distance'' of $v$ from $s$). Consider the following problem (where $\bar l$
and $\bar a$ are defined according to~\refeq{wbar}):
\begin{numitem}
\label{eq:compact_dualpprob}
    Minimize $c \cdot l$ subject to the following constraints:
    \begin{align*}
        \phi_{s}(u) - \phi_{s}(v) \le \bar a(e) + \bar l(e)
        &
        \
        \\
        \phi_{s}(v) - \phi_{s}(u) \le \bar a(e) + \bar l(e)
        &
        \quad \mbox{for each $e = uv \in EG$;} \\
        \phi_s(t)-\phi_s(s) \ge \lambda
        &
        \quad \mbox{for all $s, t \in T$, $s \ne t$.}
    \end{align*}
\end{numitem}

 \begin{lemma}
Programs~(\dualpprob) and~\refeq{compact_dualpprob} are equivalent.
\end{lemma}
Indeed, if $(l,\phi)$ is a feasible solution to \refeq{compact_dualpprob} then,
obviously, $l$ is a feasible solution to~(\dualpprob). Conversely, let $l$ be a
feasible solution to~(\dualpprob). For $v\in VG$ and $s\in T$, define $\phi_s(v)
:= \dist_\ell(s,v)$, where $\ell := \bar l + \bar a$. It is easy to check that
$(l,\phi)$ is a feasible solution to \refeq{compact_dualpprob}.

  \Xcomment{
  \begin{proof}
    We first show that \refeq{dual_ineq} holds for $l$.
    Indeed, put $\ell := \bar l + \bar a$ and consider an arbitrary $T$-path~$P$
    connecting $s$ and $t$ ($s, t \in T$, $s \ne t$).
    Summing the $\ell$-lengths of edges along~$P$,
    one gets $\ell(P) \le \phi_s(t) - \phi_s(s) \le \lambda$
    (by the inequalities in~\refeq{compact_dualpprob}).
    Since $\ell(P) = l(P) + a(P)$ this implies \refeq{dual_ineq}, as required.
    Therefore, the minimum in \refeq{compact_dualpprob} does not exceed
    the minimum in (\dualpprob).

    For the other direction, let $l$ by an optimal solution to (\dualpprob).
    One has to construct $\phi$ obeying the inequalities in \refeq{compact_dualpprob}.
    To this aim, put $\phi_s(v) := \dist_\ell(s,v)$ where $\ell := \bar l + \bar a$,
    as earlier. Then, the first two conditions in \refeq{compact_dualpprob}
    follow from the triangle inequality and then the third one
    is implied by $\phi_s(t) - \phi_s(s) = \dist_\ell(s,t) - 0 = l(P) + a(P) \ge \lambda$
    (see~\refeq{dual_ineq}).
\end{proof}
  }

  \medskip
The size of the constraint matrix in \refeq{compact_dualpprob} (written in binary
notation) is polynomial in $\abs{VG}$. Therefore, (\dualpprob) is solvable in
strongly polynomial time by use of Tardos's version of the ellipsoid method.
(This remains valid when $a$ and $\lambda$ are nonnegative rational numbers.)

\subsection{Computing the load function of an optimal multiflow}
\label{ssec:load_funct}

The following fact is of importance.
\begin{lemma}
\label{lm:node_load}
    One can find, in strongly polynomial time, a function $g \colon VG \to \Z_+$
    such that $g(v) = \hat F(v)$ for all $v \in VG$, where $F$ is some integer optimal multiflow in~(\ncpprob).
\end{lemma}
\begin{proof}
    We explain that in order to construct the desired $g$, it suffices to compare
    two optimal objective values: one for the original (integer) costs $a$ and the
    other for certain perturbed costs $a_\eps$. These values are computed by
    solving the corresponding dual problems by the method described in the previous
    subsection.

    More precisely, let $v_1,\ldots,v_n$ be the nodes of $G$. Let $U:=\max_i c(v_i)+2$,
    define $\eps(v_i):=1/U^{i+1}$, $i=1,\ldots,n$, and define the cost
    function $a_\eps$ on $VG$ to be $a+\eps$. Then for any integer feasible
    multiflow $F$, we have
    $$
        0 \le \Phi(F,a,\lambda) - \Phi(F,a_\eps,\lambda) = \sum_i \hat F(v_i) \eps(v_i) <
        U^{-1} + U^{-2} + \ldots + U^{-n} < 1.
    $$
    This and the fact that $\Phi(F,a,\lambda)$ is an integer (as $F,a,\lambda$ are
    integral) imply that if $F$ is optimal for $a_\eps$, then $F$ is optimal for
    $a$ as well. (An integer optimal multiflow for even capacities $c$ exists by
    \refth{primal_halfint}.) Moreover, for such an optimal $F$, the number
    $r := \sum_i \hat F(v_i) \eps(v_i)$ is computed in strongly polynomial time, since it is
    equal to $c\cdot l - c\cdot l_\eps$, where $l$ and $l_\eps$ are optimal dual
    solutions for $a$ and $a_\eps$, respectively. Here we use the LP duality
    equalities $\Phi(F,a,\lambda)=c\cdot l$ and $\Phi(F,a_\eps,\lambda) = c\cdot l_\eps$.
    Also the size of binary encoding of $a_\eps$ is bounded by that of $a$
    times a polynomial in $n$, so the dual problem with $a_\eps$ is solved in
    strongly polynomial time w.r.t. the original data.

    Hence, we have $r U^{n+1} = \sum_i \hat F(v_i) U^{n-i}$.
    The number $rU^{n+1}$ is an integer and, in view of $\hat F(v_i) \le c(v_i) < U$
    for each $i$, the $n$ coefficients in its base $U$ decomposition (the
    representation via degrees of $U$) are just $\hat F(v_1), \ldots, \hat F(v_n)$, thus giving~$g$.
\end{proof}

Recall that together with a node load function each multiflow $F$ also induces
its edge counterpart (see~\refeq{edge_load}). \reflm{node_load} can be
strengthened as follows.
\begin{lemma}
    \label{lm:full_load}
    One can find, in strongly polynomial time, a function $g \colon VG \cup EG \to \Z_+$
    such that $g(v) = \hat F(v)$ for all $v \in VG$ and $g(e) = \hat F(e)$ for all $e \in EG$,
    where $F$ is some integer optimal multiflow $F$ in~(\ncpprob).
\end{lemma}
\begin{proof}
    Split each edge $e = uv$ of $G$ into two edges $ux_e,x_ev$ in series and assign
    to each new node $x_e$ the capacity $c(x_e):= \min\set{c(u), c(v)}$ and the cost
    $a(x_e) := 0$. Clearly this transformation does not affect the problem in essence.
    The node load function, which can be found in strongly polynomial time by \reflm{node_load},
    yields the desired node and edge load functions in the original graph~$G$.
\end{proof}

\subsection{Constructing an optimal primal solution}

Now we explain how to find, in strongly polynomial time, an integer optimal
multiflow solving~(\ncpprob) (for a graph~$G$, even node capacites~$c$,
rational node costs $a$, and an integer parameter~$\lambda$) by using an
optimal dual solution $l$ and a function $g$ as in \reflm{full_load}. For this
$g$, there exists an integer $\ell$-geodesic multiflow $F'$ in $G$ satisfying
$F'(v) = g(v)$ for all $v \in VG$ and $F'(e) = g(e)$ for all $e \in EG$, where
$\ell := \bar a + \bar l$. Our goal is to construct one of such multiflows
explicitly.

To do this, we consider the \emph{compact} BD-graph $H$ related to $G_\ell$
(see \refssec{aux_graph}) and put $f$ to be the function on $EH$ corresponding
to $g$. More precisely, let $E'$ be the subset of edges of $H$ neither incident
to $q$ nor contained in the gadgets $\Gamma_w$ ($w \in V^\natural$). By the
construction of $H$, there is a natural bijection $\gamma$ of $E'$ to the set
$(VG_\ell-V^\natural)\cup EG_\ell$. For each $e\in E'$, we set
$f(e):=g(\gamma(e))$. In their turn, the values of $f$ on the edges of a gadget
$\Gamma_w$ are assigned as follows: for the loop $e_w$ at $\theta_w$, set
$f(e_w) := g(w)$, and for each $s\in T$, set $f(e_{w,s}):=\sum(g(e)\colon e\in
E_{s,w})$, where $E_{s,w}$ is the set of edges in $G_\ell$ connecting $V^s$ and
$w$. Finally, for each $s\in T$, we set $f(qs):=h(s)$.

Using the fact that the function $g$ on $VG\cup EG$ is determined by some
integer optimal ($\ell$-geodesic) multiflow $F'$, it follows that the obtained
function $f$ on $EH$ is integer-valued and has zero divergency at all nodes
different from $q$. So $f$ is an IBD-flow in $H$. Moreover, $f$ is generated by
$F'$ as above; in particular, $f$ is good (i.e. satisfies~\refeq{good}). By
\reflm{f-F}, we can find, in strongly polynomial time, an integer
$\ell$-geodesic multiflow $F$ generating~$f$. Then $F$ and $F'$ have the
coinciding node and edge load functions, and the optimality of $F'$ implies
that $F$ is an integer optimal solution to~(\ncpprob) as well, as required.

  \section{Proof of \refprop{aux_bidir_flow}}
\label{sec:prim_h_int}

To complete the proof of \refth{primal_halfint} it remains to prove
\refprop{aux_bidir_flow}, which claims the existence of an IBD-flow saturating the
``locked'' edges. We eliminate the lower capacity constraints (induced by the locked edges)
by reducing the claim to the existence of an IBD-flow with a certain prescribed value.

\subsection{Maximum IBD-flows} \label{ssec:IBD-flows}

Let $\Gamma$ be a bidirected graph with a distinguished source~$q$ and edge
capacities $c \colon E\Gamma \to \Z_+$, as described in \refssec{bidir_flows}.
The classic max-flow min-cut theorem states that the maximum flow value is
equal to the minimum cut capacity. A bidirected version of this theorem
involves a somewhat more complicated object, called an \emph{odd barrier}. In
this subsection we give its definition and state the crucial properties (in
Theorems~\ref{th:max_ibd_flow_min_odd_barrier},\;\ref{th:bd_residual},\;\ref{th:bd_barrier_from_flow})
that will be used in the upcoming proof of \refprop{aux_bidir_flow}. These
properties are nothing else than translations, to the language of bidirected
graphs, of corresponding properties established for integer symmetric flows in
skew-symmetric graphs, as we will explain in the Appendix.

Let $X\subseteq V\Gamma-\{q\}$. The \emph{flip} at (the nodes of) $X$ modifies
$\Gamma$ as follows: for each node $v\in X$ and each edge $e$ incident to $v$,
we reverse the direction of $e$ at $v$ (while preserving the directions of
edges at nodes in $VG-X$). For example, if $e=\vec{u}\bvec{v}$ and $u,v\in X$
then $e$ becomes $\bvec{u}\vec{v}$, and if $e=\vec{uv}$ and $u\not\in X\ni v$
then $e$ becomes $\vec{u}\bvec{v}$. BD-graphs $\Gamma$ and $\Gamma'$ are called
\emph{equivalent} if one is obtained by a flip from the other. Note that flips
do not affect bidirected walks in $\Gamma$ in essence and do not change the
maximum value of an IBD-flow in it. We will essentially use flips to simplify
requirements in the definition of odd barriers below.

Next, we employ a special notation to designate certain subsets of edges. For
$X, Y \subseteq V\Gamma$, let $[X,Y]$ denote the set of edges with one endpoint
in $X$ and the other in $Y$. We will often add arrows above $X$ and/or $Y$ to
indicate the subset of edges in $[X,Y]$ directed in one or another way. For
example, $[\bvec X, \vec Y]$ denotes the set of edges that enter both $X$ and
$Y$, ~$[\vec X, Y]$ denotes the set of edges leaving $X$ and having the other
endpoint in~$Y$ (where the direction is arbitrary), and $[\vec X, \bvec X]$
denotes the set of edges that leave $X$ at both endpoints (including twice
leaving loops). When $Y = V\Gamma - X$, the second term in the brackets may be
omitted: $[X]$, $[\vec X]$, and $[\bvec X]$ stand for $[X, V\Gamma - X]$,
$[\vec X, V\Gamma - X]$, and $[\bvec X, V\Gamma - X]$, respectively. Finally,
for a function $\phi$ on the edges, we write $\phi[X,Y]$ (rather than
$\phi([X,Y])$) for $\sum_{e \in [X,Y]} \phi(e)$.

\medskip

\begin{figure}[tb]
    \centering
    \includegraphics{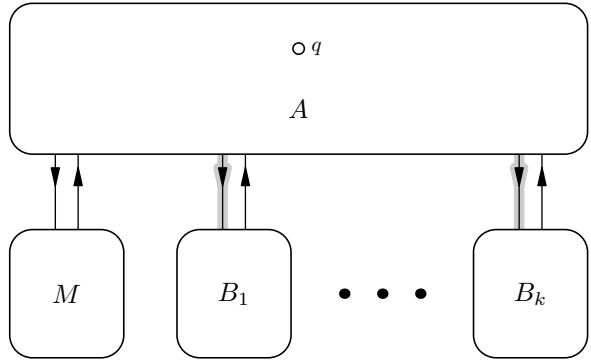}%
    \caption{
        A bidirected odd barrier.
        Grayed edges correspond to odd capacity constraints (w.r.t. $\Gamma'$).
    }
    \label{fig:bd_barrier}
\end{figure}

A tuple $\calB = (\Gamma'\,|\, A, M; B_1, \ldots, B_k)$, where $\Gamma'$ is
some BD-graph equivalent to $\Gamma$, is called an \emph{odd barrier} for
$\Gamma$ if the following conditions hold with respect to $\Gamma'$
(see~\reffig{bd_barrier}):
\begin{numitem}
\label{eq:bd_odd_barrier}
    \begin{itemize} \compact
        \item[(i)]
        $A, M, B_1, \ldots, B_k$ give a partition of $V\Gamma'=V\Gamma$, and $q \in A$.

        \item[(ii)]
        For each $i = 1, \ldots, k$, ~$c[\vec A,B_i]$ is odd.

        \item[(iii)]
        For distinct $i, j = 1, \ldots, k$, ~$c[B_i,B_j] = 0$.

        \item[(iv)]
        For each $i = 1, \ldots, k$, ~$c[B_i,M]= 0$.
    \end{itemize}
\end{numitem}
The \emph{capacity} of $\calB$ is defined to be
\begin{equation}
\label{eq:barrier_cap}
    c(\calB) := 2c[\vec A, \bvec A] + c[\vec A] - k.
\end{equation}
  \begin{theorem}[Max IBD-Flow Min Odd Barrier Theorem]
 \label{th:max_ibd_flow_min_odd_barrier}
For $\Gamma, c, q$ as above, the maximum IBD-flow value is equal to the minimum
odd barrier capacity. A (feasible) IBD-flow~$g$ and an odd barrier $\calB =
(\Gamma'\,|\, A, M; B_1, \ldots, B_k)$ for $\Gamma$ have maximum value and
minimum capacity, respectively, if and only if the following conditions hold
with respect to $\Gamma'$:
    \begin{itemize} \compact
        \item[\rm(i)]  $g[\vec A, \bvec A] = c[\vec A, \bvec A]$ and $g[\bvec A, \vec A] = 0$;
        \item[\rm(ii)] $g[\vec A,M] = c[\vec A,M]$ and $g[\bvec A,M] = 0$;
        \item[\rm(iii)] for each $i = 1, \ldots, k$, either $g[\vec A,B_i] = c[\vec A,B_i] - 1$ and
        $g[\bvec A, B_i] = 0$, or $g[\vec A,B_i] = c[\vec A,B_i]$ and $g[\bvec A,B_i] = 1$.
    \end{itemize}
\end{theorem}

Note that there may exist many minimum capacity odd barriers for $\Gamma, c,
q$. It is well-known that in a usual edge-capacitated digraph with a source $s$
and a sink $t$, the set of nodes reachable by paths from $s$ in the residual
digraph of a maximum $s$--$t$ flow $f$ determines a minimum capacity $s$--$t$
cut. Moreover, this minimum cut does not depend on the choice of $f$ and,
therefore, may be regarded as the \emph{canonical} one.

A similar phenomenon takes place for maximum IBD-flows and minimum odd barriers
(and we will essentially use this in the proof of \refprop{aux_bidir_flow}). To
describe this, consider an IBD-flow $g$ in $\Gamma$. The \emph{residual}
BD-graph $\Gamma_g$ endowed with the \emph{residual} capacity $c_g \colon
E\Gamma_g \to \Z_+$ is constructed in a similar way as for usual flows. More
precisely, $V\Gamma_g = V\Gamma$ and the edges of $\Gamma_g$ are as follows:
\begin{numitem}
\label{eq:residual_edges}
    \begin{itemize} \compact
        \item[(i)] each edge $e \in E\Gamma$ with $g(e) < c(e)$
        whose residual capacity is defined to be $c_g(e) := c(e) - g(e)$, and
        \item[(ii)]
        the \emph{reverse} edge $e^R$ to each edge $e \in E\Gamma$ with $g(e) > 0$;
        the directions of $e^R$ at the endpoints are reverse to those of~$e$ and
        the residual capacity is $c_g(e^R) := g(e)$.
    \end{itemize}
\end{numitem}

A bidirected walk $P$ in $\Gamma_g$ is called \emph{$c_g$-simple} if $P$ passes
each edge $e$ at most $c_g(e)$ times. If $P$ is a $c_g$-simple closed $q$--$q$
walk leaving its end $q$ twice, we can increase the value of $g$ in $\Gamma$
by~2, by sending one unit of flow along~$P$. So the existence (in $\Gamma_g$)
of such a walk $P$, which is called \emph{$(\Gamma,g)$-residual}, implies that
$g$ is not maximum. A converse property holds as well.
   \begin{theorem}
\label{th:bd_residual}
    An IBD-flow $g$ in $\Gamma$ is maximum if and only if there is no $(\Gamma,g)$-residual walk.
\end{theorem}

When we are given a maximum IBD-flow $g$, a certain minimum odd barrier can be
constructed by considering the residual graph $\Gamma_g$. Namely, let $\vec
R_{\Gamma,g}$ (resp. $\bvec R_{\Gamma,g}$) be the set of nodes $v$ that are
reachable by a $(\Gamma,g)$-residual $q$--$v$ walk that leaves $q$ and
enters~$v$ (resp. leaves both $q$ and $v$). Then $q \notin \bvec R_{\Gamma,g}$,
by the maximality of~$g$.

 \begin{theorem}
 \label{th:bd_barrier_from_flow}
Let $g$ be a maximum IBD-flow for $\Gamma,c,q$.  Define $A := (\vec R- \bvec R)
\cup (\bvec R- \vec R)$ and $M := V\Gamma - (\vec R\cup \bvec R)$, where $\vec
R:=\vec R_{\Gamma,g}$ and $\bvec R:=\bvec R_{\Gamma,g}$. Let $B_1, \ldots, B_k$
be the node sets of weakly connected components of the underlying undirected
subgraph of $\Gamma_g$ induced by $\vec R \cap \bvec R$. Define $\Gamma'$ to be
the BD-graph obtained from $\Gamma$ by flipping the set $\bvec R - \vec R$
(contained in $A$). Then $\calB_g := (\Gamma'\,|\, A, M; B_1, \ldots, B_k)$ is
a minimum odd barrier.
  \end{theorem}

An important fact is that the minimum odd barrier $\calB_f$ does not depend on
$g$, and we refer to it as the \emph{canonical} odd barrier for $\Gamma,c,q$.
  \begin{theorem}
  \label{th:canon_bd_barrier}
The sets $\bvec R_{\Gamma,g}$ are same for all maximum IBD-flows $g$ in
$\Gamma$, and similarly for the sets $\vec R_{\Gamma,g}$, the minimum odd
barriers $\calB_f$, and the graphs $\Gamma'$ obtained from $\Gamma$ by flipping
$\bvec R_{\Gamma,g} - \vec R_{\Gamma,g}$.
\end{theorem}

\subsection{Proof of \refprop{aux_bidir_flow}}

In fact, we have freedom of choosing any of the two forms (expensive or
compact) of~$H$ to prove \refprop{aux_bidir_flow} in full, as it is easy to see
that the claims in these cases are reduced to each other. We prefer to deal
with the expensive form, taking advantage from certain nice structural features
arising in this case. One reason for our choice is that any IBD-flow in the
expensive $H$ is automatically good, as explained in
Remark~\ref{rem:compact_vs_expensive}.

We know that there exists a good fractional bidirected $q$-flow $f$ in $H$ that
saturates the set $E_0$ of locked edges, and our goal is to show the existence
of an IBD-flow saturating $E_0$. Recall that any edge $e\in E_0$ is generated
by some node $v$ of $G$, i.e. $e=e_v$.

\medskip

It will be convenient for us to construct the desired IBD-flow without
explicitly imposing the ``lower capacities'' on the locked edges. For this purpose,
we modify $H$ as follows.

First, we add a loop $\bvec{q}\vec{q}$ with infinite capacity (entering $q$ twice).
Also we add to $H$ a node $z$, which is regarded as a new source.

Second, let $E_0$ contain an edge $e_v = \vec{v}^1 \vec{v}^2$ generated by a
vertex $v\in VG_\ell$ in some zone $V^s$, $s \in T$. We delete $e_v$ from $H$
and, instead, add two edges $\vec{v}^1\bvec{z}$ and $\vec{z}\vec{v}^2$ with
capacity $c(v)$ each.

Third, let $E_0$ contain the loops $e_{w^i}$ ($i=1,\ldots,c(w)$) for some
central node $w$ of $G_\ell$. We replace each $e_{w^i}$ (having unit capacity)
by edge $\vec{\vphantom{\theta} z}\bvec{\theta_{w^i}}$ with capacity 2; we call
it the \emph{root edge} at $\theta_{w^i}$.

We denote the resulting BD-graph by $H^1$ and keep the previous notation $c$
for its edge capacities. The above $q$-flow $f$ is transformed, in an obvious
way, into a feasible $z$-flow in $H^1$, denoted by $f$ as before. Note that
this $f$ saturates all edges created from those in $E_0$ (i.e. from $e_v$ and
$e_{w^i}$ as above); these edges leave $z$ and the value of $f$ is maximum
among the feasible $z$-flows in $H^1$ and is equal to $2c(E_0)$.

\medskip

Let $g$ be a maximum IBD-flow in $H^1$. We are going to prove that $\val{f} =
\val{g}$. This would imply that the corresponding IBD-flow in $H$ saturates
$E_0$ as required. To this aim, consider the \emph{canonical} odd barrier
$\calB = (H^2\,|\, A,M; B_1,\ldots,B_k)$ for $H^1,c,z$ (see
\refth{canon_bd_barrier}). Here $H^2$ is the BD-graph (with the source $z$)
obtained from $H^1$ according to \refth{bd_barrier_from_flow} (i.e.
$H^2:=\Gamma'$ for $\Gamma:=H^1$). From now on, speaking of edge directions,
the capacities $c$ and the flow $g$, we mean those in $H^2$, unless explicitly
stated otherwise.

We have (cf.~\refeq{barrier_cap})
  \begin{equation}
\label{eq:val_g}
    \val{g} = c(\calB) = 2c[\vec{A}, \bvec{A}] + c[\vec{A}] - k.
\end{equation}

The following assertion is crucial.
\begin{lemma}
\label{lm:b_sets}
    For each $p=1,\ldots,k$:
    \begin{itemize}
\item[\rm(i)] $B_p = \set{\theta_{w^i}}$ for some $w\in V^\natural$ and $i\in\{1, \ldots, c(w)\}$;
\item[\rm(ii)] $e_{w^i}$ is not locked (so $H^1$ contains the loop $e_{w^i}$ but not the root edge
at $\theta_{w^i}$);
\item[\rm(iii)] among the edges (legs) connecting $A$ and $B_p$, one edge leaves~$A$
and the other edges enter~$A$.
  \end{itemize}
\end{lemma}
  \begin{proof}
By the constructions of $H$ and $H^1$, for any $w\in V^\natural$ and distinct
$i, j = 1, \ldots, c(w)$, there is an automorphism $\pi=\pi_{w,i,j}$ of $H^1$
that swaps $\theta_{w^i}$ and $\theta_{w^j}$ and is invariant on the other
nodes. Also $\pi$ respects the capacities in $H^1$, and the function $\tilde g$
induced by $g$ under $\pi$ (i.e. $\tilde g(e) := g(\pi(e))$) is again a maximum
IBD-flow in $H^1$. Since $\calB$ is canonical, it follows
from~\refth{canon_bd_barrier} that
    \begin{equation}
    \label{eq:r_preserved}
\vec R_{H^1,g} = \vec R_{H^1,\tilde g}, \quad \bvec R_{H^1,g} = \bvec
R_{H^1,\tilde g}, \quad
 \vec R_{H^1,g} \cup \bvec R_{H^1,g} = A \cup B_1 \cup \ldots \cup B_k.
    \end{equation}
The nodes in $\bvec R_{H^1,g} - \vec R_{H^1,g}$ are \emph{flipped} when
constructing~$H^2$ from $H^1$. Then \refeq{r_preserved} and
\refth{bd_barrier_from_flow} imply that for $i,j$ as above,
    \begin{numitem}
    \label{eq:theta_preserved}
        \begin{itemize}
            \item[(a)] $\theta_{w^i}$ is flipped if and only if $\theta_{w^j}$
is flipped;
            \item[(b)] $\theta_{w^i}\in A$ if and only if $\theta_{w^j} \in A$.
            \item[(c)] $\theta_{w^i}\in B_1 \cup \ldots \cup B_k$ if and only
if $\theta_{w^j} \in B_1 \cup \ldots \cup B_k$.
        \end{itemize}
    \end{numitem}

Let $p \in\{1, \ldots, k\}$. Since the capacity $c[\vec A, B_p]$ (in $H^2$) is
odd, the set $[\vec A, B_p]$ contains an edge~$e$ with $c(e)$ odd. Any edge in
$H^2$ having an odd capacity is either a loop $e_{w^i}$ or a leg $e_{w^i,s}$
(regarding ``infinite'' capacities as even ones).

Obviously, no loop can be ``responsible'' for the oddness of $c[\vec A, B_p]$.

So $e = e_{w^i,s} = \theta_{w^i}\theta_{w,s}$ for some $w\in V^\natural$, $1
\le i \le c(w)$ and $s\in T$. Let $\hat e$ denote the edge of $H^1$
corresponding to $e$. Then (see~\reffig{aux_graph}(b)) $\hat e$ leaves $\theta_{w,s}$ and
enters $\theta_{w^i}$. Due to flips, however, this may not be the case for $e$
in $H^2$.

Suppose $\theta_{w^i} \in A$ (and $\theta_{w,s}\in B_p$). Then $e$ leaves
$\theta_{w^i}$, whence $\theta_{w^i}$ is a flipped node in $A$. Now
\refeq{theta_preserved}(a,b) imply that all $\theta_{w^j}$ are flipped nodes
belonging to $A$ and that $e_{w^j,s} \in [\vec A, B_p]$ for all $j = 1, \ldots,
c(w)$. But then $e$ cannot be ``responsible'' for the oddness of $c[\vec A,
B_p]$ since $c(w)$ is even.

So we have $\theta_{w,s} \in A$ and $\theta_{w^i} \in B_p$. Then $e$ leaves
$\theta_{w,s}$. The edge $\hat e$ leaves $\theta_{w,s}$ as well. Hence
$\theta_{w,s}$ is not flipped. Since $c(w)$ is even, there must be $j\in\{1,
\ldots, c(w)\}$ such that the leg $e_{w^j,s} = \theta_{w,s}\theta_{w^j}$ is not
in $[\vec A, B_p]$ (for otherwise one may pick another pair $w',i'$). Then
$\theta_{w^j}$ is not in $B_p$. In view of~\refeq{theta_preserved}(c),
$\theta_{w^j}$ belongs to a $B$-set in $\calB$ different from $B_p$.
Considering the automorphisms $\pi = \pi_{w,i',j'}$ for all distinct $i', j' =
1, \ldots, c(w)$ and using the fact that the canonical barrier $\calB$
preserves under $\pi$ (in view of~\refeq{r_preserved}), we can conclude that
the nodes $\theta_{w^1}, \ldots, \theta_{w^{c(w)}}$ belong to different
$B$-sets in $\calB$. Since these $B$-sets are pairwise disjoint and each
automorphism $\pi$ swaps two copies of $\theta_{w},$ and do not move the
remaining nodes in $H^2$, each of these $B$-sets can contain only a single
node. Thus, $B_p = \set{\theta_{w^i}}$, yielding~(i) in the lemma.
\medskip

Next we show (ii). From the construction of $H^2$ it follows that
    \begin{equation}
    \label{eq:H2}
A=\vec R_{H^2,g}-\bvec R_{H^2,\tilde g} \quad \mbox{and} \quad
 B_1 \cup \ldots \cup B_k= \vec R_{H^2,g} \cap \bvec R_{H^2,g}.
    \end{equation}
By the first equality, any $(H^2,g)$-residual walk ending at a node $v\in A$
enters $v$, and by the second equality, there exist an $(H^2,g)$-residual walk
$P$ to $\theta:=\theta_{w^i}$ that enters $\theta$ and an $(H^2,g)$-residual
walk $Q$ to $\theta$ that leaves $\theta$. Recall that the residual walks leave
the source $z$. Let $a=u\vec{\theta}$ be the last edge of $P$, and
$b=v\bvec{\theta}$ the last edge of $Q$. Define $E'$ ($E''$) to be the set of
legs $e=e_{w^i,s}$ with $g(e)=0$ (resp. $g(e)=1$). Note that
(cf.~\refeq{residual_edges})
  \begin{numitem}
if $e\in E'$ then $e^R\notin EH^2_g$ and $e$ enters $\theta$ in $H^2_g$, and
if $e\in E''$ then $e\notin EH^2_g$ and $e^R$ leaves $\theta$ in $H^2_g$.
   \label{eq:EpEpp}
   \end{numitem}

Supposing the existence of the root edge $r=\vec z\bvec{\theta}$ (in $H^1$ and
$H^2$), we can come to a contradiction as follows. Since there is no loop at
$\theta$, both nodes $u,v$ are in $A$. Note that the edge $a$ is different from
$r$ (which leaves $\theta$) and from $r^R$ (which enters $z$).
Then~\refeq{EpEpp} implies that $a\in E'$. Furthermore, $a$ is of the form
$\vec{u}\vec{\theta}$. For if $a$ enters $u$ then the edge of $P$ preceding $a$
leaves $u$, whence the part of $P$ from $z$ to $u$ forms an $(H^2,g)$-residual
walk leaving $u$, which is impossible since $u\in A$.

So $a\in[\vec A,B_p]$ and $g(a)=0=c(a)-1$. Then $g[\bvec A,B_p]=0$, by
\refth{max_ibd_flow_min_odd_barrier}(iii). This implies that $E''\subseteq[\vec
A,B]$. But then the last edge $b=v\bvec\theta$ of the walk $Q$ as above cannot
be reverse to any edge in $E''$; for otherwise $b$ enters $v$, implying that
the part of $Q$ from $z$ to $v$ leaves $v$. Also $b$ is neither reverse to an
edge in $E'$ (cf.~\refeq{EpEpp}), nor equal to $r$. The latter is because
$r\in[\vec A,B_p]$, and therefore, $g(a)<c(a)$ implies $g(r)=c(r)$ (cf.
\refth{max_ibd_flow_min_odd_barrier}(iii)), whence $r\notin EH^2_g$. Thus, $Q$
does not exist. This contradiction yields~(ii).
\medskip

It remains to show (iii). By~(ii), we have $e_{w^i}\in EH^2$, and $[A,B_p]$ is
exactly the set of legs at $\theta:=\theta_{w^i}$. Suppose $d:=|[\vec
A,B_p]|\ne 1$. Then $d\ge 3$, since $c[\vec A,B_p]=d$ is odd. Hence $g[\vec
A,B_p]\ge d-1\ge 2$ (by \refth{max_ibd_flow_min_odd_barrier}(iii)). Also the
fact that all legs enter $\theta$ together with $\div_g(\theta)=0$ and
$g(e_{w^i})\le 1$ implies that the only possible case is when $g(e_{w^i})=1$,
$g[\vec A,B_p]=2$ and $g[\bvec A,B_p]=0$. Now take an $(H^2,g)$-residual walk
$Q$ to $\theta$ that leaves $\theta$, and let $b$ be its last edge. Then $b$ is
neither the loop $e_{w^i}$ (which is saturated), nor reverse to a leg
$e=v\vec\theta$ with $g(e)>0$. Indeed, if $b=e^R$ then $b$ enters $v$ (in view
of $v\in A$ and $e\in[\vec A,B_p]$), and hence the part of $Q$ from $z$ to $v$
leaves $v$, which is impossible since $v\in A$. This contradiction yields~(iii)
and completes the proof of the lemma.
  \end{proof}

\medskip

Based on Lemma~\ref{lm:b_sets}, we now finish the proof of
\refprop{aux_bidir_flow}. Consider $B_p = \set{\theta_{w^i}}$ and let
$e_{w^i,s}$ be the unique edge in $[\vec A, B_p]$. Then $[\bvec A, B_p] =
\setst{e_{w^i,t}}{t\in T-\{s\}}$. Consider the maximum fractional BD-flow $f$
as before. By the goodness of $f$ (see \refeq{good}), we have
  \begin{multline*}
    f[\vec A, B_p] - f[\bvec A, B_p] = f(e_{w^i,s}) -
                     \sum\nolimits_{t\in T-\{s\}} f(e_{w^i,t}) \\
    = f(e_{w^i,s}) - \left( 2f(e_{w^i}) - f(e_{w^i,s}) \right) =
  2 \left( f(e_{w^i,s}) - f(e_{w^i}) \right)  \le 0 = c[\vec A, B_p] - 1.
  \end{multline*}

Using this and \refeq{val_g}, we have
  \begin{multline*}
\val{f} = \div_f(z) = \sum\nolimits_{v \in A} \div_f(v) =
    \left( 2f[\vec A, \bvec A] - 2f[\bvec A, \vec A] \right) +
                \left( f[\vec A] - f[\bvec A] \right) \\
= \left( 2f[\vec A, \bvec A] - 2f[\bvec A, \vec A] \right) +
          \left( f[\vec A,M] - f[\bvec A, M] \right) +
         \sum\nolimits_{p=1}^k \left( f[\vec A, B_p] - f[\bvec A, B_p] \right) \\
  \le 2c[\vec A, \bvec A] + c[\vec A, M] + \sum\nolimits_{p=1}^k \left( c[\vec A,B_p] - 1 \right)
         = c[\vec A] + 2c[\vec A, \bvec A] - k = c(\calB) = \val{g}.
  \end{multline*}
Thus, we obtain the desired relation $\val{f} \le \val{g}$ (which, in fact,
holds with equality). This completes the proof of \refprop{aux_bidir_flow}.

\section{Dual half-integrality}

\subsection{Polyhedral approach}

\begin{theorem}
 \label{th:dual_halfint}
Let $a:VG\to\Z_+$ and $p \in \Z_+$. Then problem (\dualpprob) has a
half-integer optimal solution.
 \end{theorem}
 \begin{proof}
The proof follows easily from \refth{primal_halfint} and the general fact that
the ``totally dual $1/k$-integrality'' implies the ``totally primal
$1/k$-integrality'', which is a natural generalization of a well-known result
on TDI-systems due to Edmonds and Giles~\cite{EG-77}. More precisely, we
utilize the following simple fact (see, e.g., \cite[Statement 1.1]{kar-89}):
    \begin{lemma} \label{lm:dual-primal}
Let $A$ be a nonnegative $m \times n$-matrix, $b$ an integral $m$-vector, and
$k$ a positive integer. Suppose that the program $D(c) := \max \setst{yb}{y \in
\Q^m_+,~ yA \le c}$ has a $1/k$-integer optimal solution for every nonnegative
integral $n$-vector $c$ such that $D(c)$ has an optimal solution. Then for
every nonnegative integral $n$-vector $c$, the program $P(c) := \min
\setst{cx}{x \in \Q^n_+,~ Ax \ge b}$ has a $1/k$-integer optimal solution
whenever it has an optimal solution.
    \end{lemma}
In our case, we set $k := 2$ and take as $A$ (resp. $b$) the constraint matrix
(resp. the right hand side vector) of~(\dualpprob). Then $b$ is integral,
$D(c)$ becomes~(\ncpprob), $P(c)$ becomes~(\dualpprob), and the
half-integrality for the former implies that for the latter.
  \end{proof}

This proof is not ``constructive'' and does not lead directly to an efficient
method for finding a half-integer optimal solution $l$ to~(\dualpprob). Below
we devise a strongly polynomial algorithm.

\subsection{The algorithm}

It has as the input arbitrary (rational-valued) optimal solutions $l$ and $F$
to~(\dualpprob) and~(\ncpprob), respectively, and outputs a half-integer
optimal solution~$\hat l$ to~(\dualpprob). (Such $l$ and $F$ can be found in
strongly polynomial time as described in \refsec{primal_alg}.)

As before, we set $\ell:=\bar a+\bar l$, and in what follows, speaking of a
geodesic, we always mean an $\ell$-geodesic in $G$, i.e. a $T$-path $P$ with
$\ell(P) = a(P) + l(P) = \lambda$. Our goal is to construct
$\hat l \colon VG \to \frac12 \Z_+$ satisfying the following conditions:
  \begin{numitem}
 \label{eq:hat_l}
    \begin{itemize}
\item[\rm (i)] $a(P) + \hat l(P) \ge \lambda$ for any $T$-path~$P$ in $G$;

\item[\rm (ii)] $a(P) +\hat l(P) = \lambda$ for each geodesic~$P$;

\item[\rm (iii)] for $v \in VG$, if $l(v) = 0$ then $\hat l(v) = 0$.
    \end{itemize}
  \end{numitem}
Then~\refeq{hat_l} and the complementary slackness conditions
\refeq{cs_geodesic}--\refeq{cs_saturated} imply that the node length $\hat l$
forms an optimal solution to (\dualpprob).

\medskip

We construct an undirected graph $\Gamma$ and endow it with \emph{edge} lengths
$\mu \colon E\Gamma \to \Z_+$ as follows. We first include in $\Gamma$ the
terminal set $T$ and all nodes and edges of $G$ contained in geodesics. Also we
add to $\Gamma$ the edges of $G$ with both ends lying on geodesics or $T$. The
edges $e$ of the current $\Gamma$ are called \emph{regular} and we define
$\mu(e):= 0$.

Next we add to $\Gamma$ additional edges, which are related to constraints due
to parts of $G$ outside $\Gamma$. More precisely, we scan all pairs of nodes
$u, v \in V\Gamma$ not connected by a (regular) edge and such that there exists
a path~$Q$ in $G$ having all nodes in $VG - V\Gamma$ and whose first node is
adjacent to $u$, and the last node to $v$. We add to $\Gamma$ edge $e = uv$,
referring to it as a \emph{virtual} edge, and define its length $\mu(e)$ to be
the minimum value of $a(Q)$ among such paths $Q$. The construction of
$\Gamma,\mu$ reduces to a polynomial number of usual shortest paths problems in
$G$.

Note that $l(v) = 0$ holds for each node $v \in VG - V\Gamma$ (by
\refeq{cs_geodesic} and \refeq{cs_saturated}). We assign $\hat l(v) := 0$ for
these nodes $v$ and will further focus on finding values of $\hat l$ on the
nodes in $\Gamma$.

For a path $P$ in $\Gamma$, let $\Lambda(P)$ denote its \emph{full length}
$a(P) + l(P) + \mu(P)$. Clearly $\Lambda(P) \ge \lambda$ holds for any
$T$-path~$P$ in~$\Gamma$, and for each $T$-path~$Q$ in~$G$, there exists a
\emph{shortcut} path~$P$ in $\Gamma$ such that $\Lambda(P) \le a(Q) + l(Q)$.

The desired lengths $\hat l$ on $V\Gamma$ will be extracted from a system of
linear constraints described below. For a node $v \in V\Gamma$, let $T_v$
($\Pi_v$) denote the set of terminals $s \in T$ (resp. pairs $s,t\in T$) such
that $v$ belongs to a geodesic from $s$ (resp. connecting $s$ and $t$). When a
terminal $s$ belongs to no geodesic, we set by definition $T_s := \set{s}$. For each
$v \in V\Gamma$ and $s \in T_v$, we introduce two variables $\rho_s^-(v)$ and
$\rho_s^+(v)$ and impose the following constraints:
 \begin{numitem}
\label{eq:rho_system}
    \begin{itemize}
\item[\rm (i)] For each $s \in T$, ~$\rho_s^-(s)  = 0$.
\item[\rm (ii)] For each $v \in V\Gamma$ and $s \in T_v$,
        \begin{eqnarray*}
            \rho_s^+(v) - \rho_s^-(v) & = a(v) & \quad \mbox{if ~$l(v) = 0$}, \\
                                               & \ge a(v) & \quad \mbox{if ~$l(v) > 0$}.
        \end{eqnarray*}
\item[\rm (iii)] For each $v \in V\Gamma$ and $\{s, t\} \in \Pi_v$,
            ~$\rho_s^+(v) + \rho_t^-(v) = \lambda$ (and $\rho_t^+(v) + \rho_s^-(v)=\lambda$).
\item[\rm (iv)] If $e = uv \in E\Gamma$ and $s \in T_u \cap T_v$, then
        \begin{eqnarray*}
            \rho_s^-(v) - \rho_s^+(u) & \le \mu(e), \\
            \rho_s^-(u) - \rho_s^+(v) & \le \mu(e).
        \end{eqnarray*}
Moreover, if there exists a geodesic from $s$ containing both $u,v$ in this
order (resp. in the order $v,u$), then the former (resp. the latter) inequality
is replaced by equality. (Note that in this case $\mu(e)=0$.)
\item[\rm (v)] If $e = uv \in E\Gamma$, $s \in T_u$, $t \in T_v$, and $s \ne t$,
then $\rho_s^+(u) + \rho_t^+(v)\ge \lambda-\mu(e)$.
    \end{itemize}
\end{numitem}

The meaning of these variables becomes evident from the proof of the next
statement.
 \begin{lemma}
    System~\refeq{rho_system} has a solution.
\end{lemma}
 \begin{proof}
For a $p$--$q$ path $P$ in $\Gamma$, define its \emph{pre-length} to be
$\Lambda(P) - (a(q) + l(q))$ (i.e. compared with the full length, we do not
count the last node). For $v \in V\Gamma$ and $s \in T_v$, define $\rho_s^-(v)$
(resp. $\rho_s^+(v)$) to be the minimum pre-length (resp. the minimum full
length) of an $s$--$v$ path in~$\Gamma$. Then~\refeq{rho_system}(i)--(iii)
follow from the construction. Condition~\refeq{rho_system}(iv) represents a
sort of triangle inequalities (giving one equality if $e$ belongs to a geodesic
from $s$). Finally, condition~\refeq{rho_system}(v) holds since the full length
of any $s$--$t$ path in $\Gamma$ is at least $\lambda$.
 \end{proof}

We observe that in linear system \refeq{rho_system}, each constraint contains
at most two variables, each occurring with the coefficient 1 or --1, and that
the R.H.S. in it is an integer. A well-known fact is that a linear system with
such features is totally dual half-integral; therefore, it has a half-integer
basis solution (whenever it has a solution at all), and such a solution can be
found in strongly polynomial time (cf., e.g.,~\cite{EJ-70,sch-03}).
\medskip

Given a half-integer solution $(\rho^-,\rho^+)$ to~\refeq{rho_system}, we
define half-integer node lengths $\hat l$ as follows:
  $$
    \hat l(v) := \rho_s^+(v) - \rho_s^-(v) - a(v) \quad
    \mbox{for all $v \in V\Gamma$ and $s \in T_v$}.
  $$

Now the desired algorithmic result is provided by the following

 \begin{lemma}
$\hat l$ is well-defined and satisfies~\refeq{hat_l}.
 \end{lemma}
 \begin{proof}
We first show that for any $v \in V\Gamma$ and $s, t \in T_v$,
    \begin{equation}
    \label{eq:same_hat_l}
        \rho_s^+(v) - \rho_s^-(v) = \rho_t^+(v) - \rho_t^-(v).
    \end{equation}
This is trivial when $\Pi_v=\emptyset$ (since in this case $v\in T$ and
$T_v=\{v\}$). Let $\Pi_v\ne\emptyset$. If $\{s,t\}\in\Pi_v$,
then~\refeq{same_hat_l} follows from~\refeq{rho_system}(iii).
Now~\refeq{same_hat_l} with any two $s,t\in T_v$ is implied by the fact that
the graph whose nodes and edges are the elements of $T_v$ and $\Pi_v$,
respectively, is connected (as it is easy to see that for
$\{s,t\},\{p,q\}\in\Pi_v$, at least one of $\{s,p\},\{s,q\}$ is in $\Pi_v$ as
well). So $\hat l$ is well-defined.

Property~\refeq{hat_l}(iii) is immediate from~\refeq{rho_system}(ii).

To see~\refeq{hat_l}(ii), consider an $s$--$t$ geodesic $P$. Going along $P$
step by step and applying~\refeq{rho_system}(ii),(iv), we observe that for each
node $v$ on $P$, the $s$--$v$ part $P'$ of $P$ satisfies $a(P') + \hat l(P') =
\rho_s^+(v)-\rho_s^-(s)$. When reaching $t$, we obtain $a(P) + \hat l(P) =
\rho_s^+(t)-\rho_s^-(s)$, and now~\refeq{hat_l}(ii) follows
from~\refeq{rho_system}(iii) and $\rho_s^-(s)=\rho_t^-(t)=0$
(by~\refeq{rho_system}(i)).

Finally, consider an arbitrary $T$-path $Q$ in $\Gamma$, from $p$ to $q$ say.
To conclude with~\refeq{hat_l}(i), it suffices to show that
  \begin{equation} \label{eq:hatLambda}
\Delta(Q):=a(Q)+\hat l(Q)+\mu(Q)\ge \lambda.
  \end{equation}

Represent $Q$ as the concatenation $Q'\cdot Q''$, where $Q'$ is a part of a
geodesic from $p$. We prove~\refeq{hatLambda} by induction on the number
$|Q''|$ of edges in $Q''$. When $|Q''|=0$, ~$Q$ is a geodesic, and we are done.
Assuming this is not the case, take the first edge $e=uv$ of $Q''$, where $u$
is the end of $Q'$. By reasonings above, $\Delta(Q')=\rho_p^+(u)$. If $v\in T$
(and therefore, $v=q$), ~\refeq{hatLambda} immediately follows
from~\refeq{rho_system}(v) (with $s:=q$ and $t:=q$). And if $v\notin T$, then
$v$ belongs to some $s$--$t$ geodesic $L$. W.l.o.g., one may assume that $s\ne
p$ and $t\ne q$. Applying~\refeq{rho_system}(v) to $s,p,e$, we have
  $$
  \rho^+_p(u)+\rho^+_s(v)+\mu(e)\ge \lambda.
  $$
Comparing this with $\rho^-_s(v)+\rho^+_t(v)=\lambda$ and using
$\rho^+_s(v)-\rho^-_s(v)=a(v)+\hat l(v)$, one can conclude that
$\Delta(Q)\ge\Delta(R)$, where $R$ is the $t$--$q$ path being the concatenation
of the $t$--$v$ part of (the reverse of) $L$ and the $v$--$q$ part $R''$ of
$Q$. Since $|R''|=|Q''|-1$, we can apply induction and obtain
$\Delta(Q)\ge\Delta(R)\ge\lambda$, as required.
  \end{proof}


\section*{Appendix: Skew-symmetric graphs and flows}

\refstepcounter{section}

In this section we recall the notions of skew-symmetric graphs and integer
skew-symmetric flows, review known results on such graphs and flows, and then
use them to derive necessary results on bidirected graphs and flows to which we
appealed in \refsec{prim_h_int}.

\subsection{Skew-symmetric graphs}
\label{ssec:sk_graph}

A \emph{skew-symmetric graph}, or an \emph{SK-graph} for short, is a digraph
$G=(VG,EG)$, with possible parallel arcs, endowed with two bijections
$\sigma_V, \sigma_A$ such that: $\sigma_V$ is an involution on the nodes (i.e.
$\sigma_V(v)\ne v$ and $\sigma_V(\sigma_V(v)) = v$ for each node~$v$);
$\sigma_A$ is an involution on the arcs; and for each arc $a$ from $u$ to $v$,
$\sigma_A(a)$ is an arc from $\sigma_V(v)$ to $\sigma_V(u)$. For relevant
results on SK-graphs and a relationship between SK- and BD-graphs, see
\cite{tut-67,GK-96,GK-04,BK-07}. For brevity $\sigma_V$ and $\sigma_A$ are
combined into one mapping $\sigma$ on $VG\cup AG$, which is called the
\emph{symmetry} (or skew-symmetry, to be precise) of $G$. For a node (arc) $x$,
its symmetric node (arc) $\sigma(x)$ is also called the \emph{mate} of $x$, and
we usually use notation with primes for mates, denoting $\sigma(x)$ by $x'$.
Although $G$ is allowed to contain parallel arcs, when it is not confusing, an
arc from $u$ to $v$ may be denoted as $(u,v)$ or $\vec{uv}$.

Observe that if $G$ contains an arc $a$ from a node $v$ to its mate $v'$, then
$a'$ is also an arc from $v$ to $v'$ (i.e. $a'$ is parallel to $a$).

The symmetry $\sigma$ is extended in a natural way to walks, subgraphs and
other objects in~$G$. In particular, two walks are symmetric to each other if
the elements of one of them are symmetric to those of the other and go in the
reverse order: for a walk $P = (v_0, a_1, v_1, \ldots, a_k, v_k)$, the
symmetric walk $P'=\sigma(P)$ is $(v'_k, a'_k, v'_{k-1}, \ldots, a'_1, v'_0)$.

Next we explain a relationship between skew-symmetric and bidirected graphs.
Given an SK-graph $G$, choose an arbitrary partition $\pi = (V_1, V_2)$ of $VG$
such that $V_2=\sigma(V_1)$. Then $G$ and $\pi$ determine the BD-graph $G^*$
with $VG^* = V_1$ whose edges correspond to the pairs of symmetric arcs in $G$.
More precisely, arc mates $a,a'$ of $G$ generate one edge $e$ of $G^*$
connecting nodes $u, v \in V_1$ such that: (i) $e$ goes from $u$ to $v$ if one
of $a, a'$ goes from $u$ to $v$ (and the other goes from $v'$ to $u'$ in
$V_2$); (ii) $e$~leaves both $u,v$ if one of $a,a'$ goes from $u$ to $v'$ (and
the other from $v$ to $u'$); (iii) $e$~enters both $u, v$ if one of $a,a'$ goes
from $u'$ to $v$ (and the other from $v'$ to $u$). Note that $e$ becomes a loop
if $a, a'$ connect a pair of symmetric nodes.

Conversely, a BD-graph $G^*$ determines an SK-graph $G$ with symmetry $\sigma$
as follows. Make a copy $\sigma(v)$ of each element $v$ of $V^* := VG^*$,
forming the set $(V^*)' := \setst{\sigma(v)}{v\in V^*}$. Put $VG:= V^* \sqcup
(V^*)'$. For each edge $e$ of $G^*$ connecting nodes $u$ and $v$, assign two
``symmetric'' arcs $a, a'$ in $G$ so as to satisfy (i)--(iii) above (where $u'
= \sigma(u)$ and $v' = \sigma(v)$). An example is depicted in~\reffig{bd-sk}.

\begin{figure}[tb]
    \centering
    \subfigure[BD-graph~$G^*$.]{
      \includegraphics{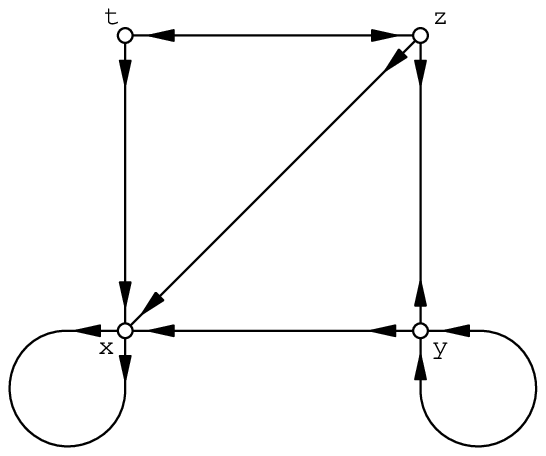}%
    }
    \hspace{1cm}%
    \subfigure[Corresponding SK-graph~$G$.]{
      \includegraphics{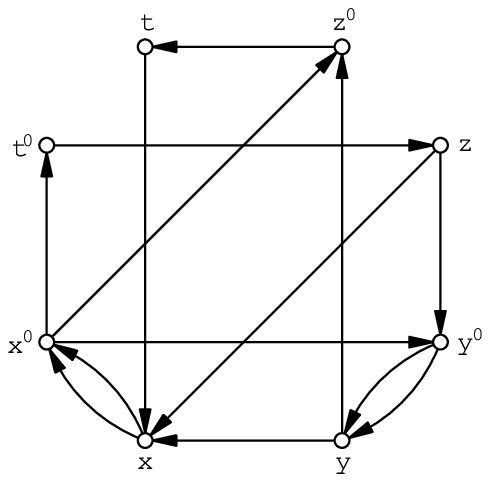}%
    }
    \caption{Related bidirected and skew-symmetric graphs.}
    \label{fig:bd-sk}
\end{figure}

 \begin{remark}
 \label{rem:bd_ss}
    Note that one BD-graph generates one SK-graph, by the second construction. On
    the other hand, one SK-graph generates a set of BD-graphs, depending on the
    partition $\pi$ of $V$, by the first construction. Namely, for each pair
    of symmetric mates $\set{v,v'}$ in $G$ one may distribute $v,v'$ between $V_1,V_2$
    so that either $v\in V_1$, $v'\in V_2$ or, reversely, $v \in V_2$, $v' \in V_1$.
The resulting BD-graphs are obtained from one other by making corresponding
flips (defined in \refssec{IBD-flows}).
 \end{remark}

There is essentially a one-to-one correspondence between the walks in $G^*$
and~$G$. More precisely, let $\tau$ be the natural mapping of $VG \cup AG$ to
$VG^* \cup EG^*$. Each walk $P = (v_0, a_1, v_1, \ldots, a_k, v_k)$ in
$G$ (where $a_i=(v_{i-1},v_i)$) induces the sequence
 $$
   \tau(P) := (\tau(v_0), \tau(a_1), \tau(v_1), \ldots,
 \tau(a_k), \tau(v_k))
$$
of nodes and edges in $G^*$. One can see that $\tau(P)$ is a walk in $G^*$
(i.e. $\tau(a_i),\tau(a_{i+1})$ form a transit pair at $\tau(v_i)$, for each
$i$) and that $\tau(P')$ is the walk reverse to $\tau(P)$. Moreover, for any
walk $P^*$ in $G^*$, there is exactly one walk $P$ in $G$ such that $\tau(P) =
P^*$ (considering $P$ up to replacing an arc $a\in AP$ by its mate $a'$ when
$a,a'$ are parallel, i.e. correspond to a loop in $G^*$).

\subsection{Skew-symmetric flows}
\label{ssec:sk_flow}

We call a function $\phi$ on the arcs of an SK-graph $G$
\emph{(self-)symmetric} if $\phi(a) = \phi(a')$ for all $a \in AG$. Let $s \in
VG$ be a designated \emph{source}; its mate $s'$ is regarded as the
\emph{sink}. An \emph{integer skew-symmetric $s$--$s'$ flow}, or an
\emph{ISK-flow} for short, is a symmetric function $f \colon AG \to \Z_+$ being
an $s$--$s'$ flow in a usual sense: $\div_f(v) = 0$ for all $v \in VG -
\set{s,s'}$, and $\div_f(s)\ge 0$. The \emph{value} of $f$ is $\val{f} :=
\div_f(s)$. Here $\div_f(v)$ denotes the usual divergence (given by
\refeq{div}, where $\deltain(v)$ and $\deltaout(v)$ are the sets of arcs
entering and leaving $v$, respectively).

For a capacity function $c \colon AG \to \Z_+$, a flow~$f$ is said to be
\emph{feasible} if $f(a) \le c(a)$ for all $a \in AG$. We refer to a feasible
ISK-flow of maximum possible value as a \emph{maximum ISK-flow}.

The above correspondence between BD- and SK-graphs is naturally extended to
flows. More precisely, if $f$ is a symmetric $s$--$s'$ flow in $G$, then
transferring the values of~$f$ from the pairs of arc mates of~$G$ to the edges
of the BD-graph~$G^*:=\tau(G)$, we obtain a $\tau(s)$-flow in~$G^*$, denoted as
$f^*$. The converse correspondence is evident as well.

For $X, Y \subseteq VG$, let $(X,Y)$ denote the set of arcs going from $X$ to
$Y$. Also (accommodating notation from \refsec{prim_h_int} to digraphs) we
denote by $[\vec X]$ the set of arcs leaving~$X$.

Let $c:AG\to\Z_+$ be a symmetric capacity function. Then a tuple $\calB = (A,
M; B_1, \ldots, B_k)$ of subsets of $VG$ is called a (skew-symmetric) \emph{odd
barrier} (w.r.t. the source~$s$) if the following conditions hold
(see~\reffig{odd_barrier}):
  \begin{numitem}
\label{eq:sk_odd_barrier}
    \begin{itemize} \compact
        \item[(i)]
        the sets $A, A'=\sigma(A), M, B_1, \ldots, B_k$ give a partition of $VG$,
        each $B_i$ is self-symmetric ($B_i' = B_i$),
        and $s \in A$;

        \item[(ii)]
        For each $i$, ~$c(A,B_i)$ is odd.

        \item[(iii)]
        For distinct $i,j$, ~$c(B_i,B_j) = 0$.

        \item[(iv)]
        For each $i$, ~$c(B_i,M) = c(M,B_i) = 0$.
    \end{itemize}
\end{numitem}
The \emph{capacity} of $\calB$ is defined to be
\begin{equation}
\label{eq:sk_barrier_cap}
    c(\calB) := c[\vec A] - k.
\end{equation}

\begin{figure}[tb]
    \centering
    \includegraphics{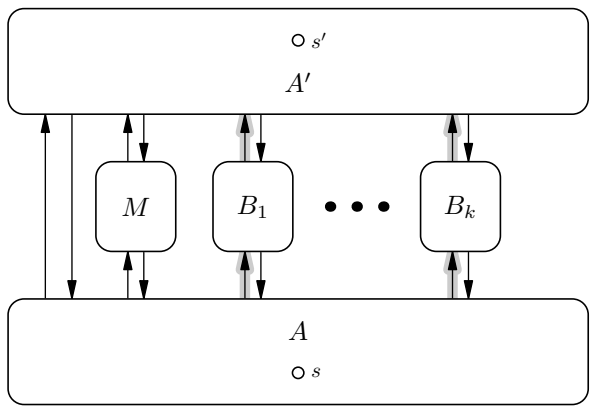}%
    \caption{
        A skew-symmetric odd-barrier.
        Grayed arcs correspond to odd capacity constraints.
    }
    \label{fig:odd_barrier}
\end{figure}

Odd barriers in skew-symmetric graphs are related to their bidirected
counterparts introduced in \refsec{prim_h_int}. Indeed, consider a BD-graph
$G^*$ with integer edge capacities $c \colon EG^* \to \Z_+$ and a source~$s$.
Construct the related SK-graph~$G$ with $VG=V\sqcup V'$, where $V:=VG^*$. Edge
capacities $c$ in $G^*$ induce symmetric arc capacities in~$G$, also denoted
by~$c$. The source $s$ in $G^*$ gives the source $s$ and the sink $s'$ in $G$.
Consider a skew-symmetric odd barrier $\calB = (A, M; B_1, \ldots, B_k)$ in
$G$.

This barrier gives rise to the following odd BD-barrier $\calB^*$ in $G^*$
obeying $c(\calB^*) = c(\calB)$. We first construct a new BD-graph from $G$ by
taking a bipartition $(V_1,V_2=\sigma(V_1))$ of $VG$ such that $A \subseteq
V_1$ and $V_1-(A\cup A')=V-(A\cup A')$; cf.~\refrem{bd_ss}. The resulting
BD-graph $H^*$ is equivalent to $G^*$. Moreover, $H^*$ is obtained from $G^*$
by flipping a subset of nodes within $A$.

The node subsets $M, B_1, \ldots, B_k$ in $G$ are self-symmetric and induce
subsets $M^*, B_1^*, \ldots, B_k^*$ in $G^*$ and $H^*$ in a natural way;
namely, $M^*:=M\cap V=M\cap V_1$ and similarly for $B_i^*$. Define $\calB^* :=
(H^* | A^*, M^*; B_1^*, \ldots, B_k^*)$, where $A^*:=(A\cup A')\cap V$. A
straightforward examination shows that the properties in~\refeq{sk_odd_barrier}
imply their bidirected counterparts in~\refeq{bd_odd_barrier}. To see that
$c(\calB^*) = c(\calB)$, define $Z := M \cup B_1 \cup \ldots \cup B_k$. Note
that $c[\vec A] = c(A, A') + c(A, Z)$. The capacity $c(A,A')$ is equal to
$2c[\vec {A^*},\bvec{A^*}]$ (in $H^*$) since $(A, A')$ consists of pairs of arc
mates, each pair corresponding to an edge in $[\vec {A^*},\bvec{A^*}]$. And the
capacity $c(A,Z)$ is equal to $c[\vec{A^\ast}]$ since the (symmetric) set~$Z$
corresponds to $M^* \cup B_1^* \ldots B_k^*$ in $H^*$.

\medskip

In light of these observations, \refth{max_ibd_flow_min_odd_barrier} is a
consequence of the following Tutte's theorem. (For shorter proofs of this and
next theorems, see also~\cite{GK-04}.)
\begin{theorem}[\rm Max ISK-Flow Min Odd Barrier Theorem \cite{tut-67}]
\label{th:max_isk_flow_min_odd_barrier}
    For $G,c,s$ as above, the maximum ISK-flow value is equal to the minimum odd
    barrier capacity. An ISK-flow $f$ and an odd barrier
    $\calB = (A, M; B_1, \ldots,B_k)$ have maximum value and minimum capacity, respectively, if and only
    if the following hold:
    \begin{itemize}
        \item[\rm(i)] $f(A,A'\cup M) = c(A, A'\cup M)$ and $f(A'\cup M, A) = 0$;
        \item[\rm(ii)] for each $i = 1, \ldots, k$, either $f(A, B_i) = c(A, B_i) - 1$ and
        $f(B_i, A) = 0$, or $f(A, B_i) = c(A, B_i)$ and $f(B_i, A) = 1$.
    \end{itemize}
\end{theorem}

\medskip

Next we establish additional correspondences. Consider an ISK-flow $f$ in $G$.
The \emph{residual} SK-graph $G_f$ endowed with the \emph{residual} capacities
$c_f \colon AG \to \Z_+$ is constructed in a standard fashion: $VG_f = VG$, and
the arcs of $G_f$ are:
\begin{numitem}
\label{eq:residual_arcs}
    \begin{itemize} \compact
        \item[(i)] each arc $a \in AG$ with $f(a) < c(a)$
        whose residual capacity is defined to be $c_f(a) := c(a) - f(a)$, and
        \item[(ii)]
        the \emph{reverse} arc $a^R = (v,u)$ to each arc $a = (u,v) \in AG$ with $f(a) > 0$;
        its residual capacity is $c_f(a^R) := f(a)$
    \end{itemize}
\end{numitem}
(cf.~\refeq{residual_edges}). A path $P$ in $G_f$ is called
\emph{$c_f$-regular} if $c_f(a)=c_f(a') \ge 2$ holds for each pair of arc mates
$a, a'$ occurring in $P$. (In other words, the bidirected image of $P$ in
$G^*_f$ is a $c_f$-simple walk.) If $P$ is a $c_f$-regular $s$--$s'$ path, we
can increase the value of $f$ by~2 (by sending one unit of flow along $P$ and
one unit of flow along $P'$). So the existence of such a $P$ implies the
non-maximality of $f$. A converse property is valid as well.
 \begin{theorem}[\rm \cite{tut-67}]  \label{th:sk_residual}
An ISK-flow $f$ is maximum if and only if there is no $c_f$-regular $s$--$s'$
path in $G_f$.
 \end{theorem}
This implies \refth{bd_residual} for IBD-flows.

Given a maximum ISK-flow $f$, a certain minimum odd barrier can be constructed
by considering the residual graph $G_f$. The construction described in the
proof of Theorem~3.5 in~\cite{GK-04} (relying on Lemma~2.2 in~\cite{GK-96}) is
as follows.
 \begin{theorem} \label{th:sk_barrier_from_flow}
Let $f$ be a maximum ISK-flow. Let $R = R_f$ be the set of nodes reachable from
$s$ by $c_f$-regular paths in $G_f$. Define $A := R - R'$ and $M := VG - (R
\cup R')$. Let $B_1, \ldots, B_k$ be the node sets of weakly connected
components of the subgraph $G$ induced by $R \cap R'$. Then $\calB_f := (A, M;
B_1, \ldots, B_k)$ is a minimum odd barrier. \qed
   \end{theorem}

This subset $R$ of nodes in $G$ corresponds to two sets $\vec R=\vec
R_{G^*,f^*}$ and $\bvec R=\bvec R_{G^*,f^*}$ in $G^*$ (defined just
before~\refth{bd_barrier_from_flow}; here $\Gamma=G^*$ and $g=f^*$). More
precisely, assuming that each node $v \in VG^*$ corresponds to node mates $v,
v'$ in $G$ (cf. \refssec{sk_graph}), one can realize that $\vec R$ (resp.
$\bvec R$) is the set of nodes $v\in VG^*$ such that $v\in R$ (resp. $v'\in
R$).

Finally, the last theorem in~\refssec{IBD-flows} (\refth{canon_bd_barrier}) is
implied by the following assertion.
  \begin{theorem} \label{th:canon_sk_barrier}
The sets $R_f$ in \refth{sk_barrier_from_flow} are equal for all maximum
ISK-flows~$f$. Therefore, the minimum odd barriers $\calB_f$ are equal as well.
  \end{theorem}
This fact can be extracted from reasonings in~\cite{GK-04}, yet it is not
formulated there explicitly. For this reason, we give a direct proof.

Let $f$ be a maximum ISK-flow such that the set $R_f$ is inclusion-wise minimal
and let $\calB_f := (A, M; B_1, \ldots, B_k)$. Consider another maximum
ISK-flow $g$ (if any). Then
   \begin{multline} \label{eq:canon_value}
   c[\vec A] - k = c(\calB) = \val{g} = \sum\nolimits_{v \in A} \div_{g}(v) \\
   = g\langle A,A'\rangle + g\langle A,M\rangle + g\langle A,B_1\rangle +
   \ldots + g\langle A,B_k\rangle,
  \end{multline}
where for disjoint subsets $X,Y\subset VG$, ~$g\langle X,Y\rangle$ denotes
$g(X,Y) - g(Y,X)$. For $i = 1, \ldots, k$, we have: $g(A, B_i) \le c(A, B_i)$;
~$c(A, B_i)$ is odd; and $g\langle A,B_i\rangle$ is even (the latter is due to
a result in~\cite{tut-67}; see also~\cite[Corollary~3.2]{GK-04}). Therefore,
$c(A,B_i) - g\langle A,B_i\rangle \ge 1$. Also $g\langle A,A'\rangle \le c(A,
A')$ and $g\langle A,M\rangle \le c(A,M)$. Comparing these relations
with~\refeq{canon_value}, we conclude that:
\begin{itemize} \compact
    \item[(i)] all arcs in $(A, A'\cup M)$ are saturated by $g$, while all arcs $a$ in
    $(A'\cup M,A)$ are free of $g$ (i.e. $g(a)=0$);

    \item[(ii)] for each $i$, ~$g\langle A,B_i\rangle = c(A,B_i) - 1$.
\end{itemize}

In terms of the residual graph $G_g$, (i) and (ii) mean that the sets $A$ and
$VG - A$ are connected in $G_g$ by exactly $k$ arcs $a_1,\ldots,a_k$, each
$a_i$ going from $A$ to $B_i$ and having the residual capacity~1. By symmetry,
$A'$ and $VG-A'$ are connected by only the arcs $a'_1, \ldots, a'_k$ (each
$a'_i$ goes from $B_i$ to $A'$, and $c_g(a_i) = 1$). Also by
\refeq{sk_odd_barrier}(iii),(iv), no arc in $G_g$ connects different sets among
$M, B_1, \ldots, B_k$. Therefore, the set $R_g$ of nodes in $G_g$ reachable
from $s$ by $c_g$-regular paths is contained in $R_f$. By the minimality of
$R_f$, we have $R_f = R_g$, as required. \qed

\bigskip
\textbf{Acknowledgements.} We thank the anonymous referee for useful
suggestions.

\end{document}